\documentclass[conference]{IEEEtran}
\ifCLASSINFOpdf
\else
\fi

\usepackage{amsmath,amsfonts,amsthm,amssymb}
\usepackage{algorithm}
\usepackage{algpseudocode}
\usepackage{graphicx}
\newtheorem{Theorem}{Theorem}

\newtheorem{Lemma}{Lemma}

\begin{document}
%
\title{Distributed Low Rank Approximation of Implicit Functions of a Matrix}


\author{\IEEEauthorblockN{David P.~Woodruff}
\IEEEauthorblockA{IBM Almaden Research Center\\
San Jose, CA 95120, USA\\
Email: dpwoodru@us.ibm.com}
\and
\IEEEauthorblockN{Peilin Zhong}
\IEEEauthorblockA{Institute for Interdisciplinary Information Sciences\\
Tsinghua University, Beijing 100084, China\\
Email: zpl12@mails.tsinghua.edu.cn}
}


%


\maketitle

\begin{abstract}
We study distributed low rank approximation in which the matrix to be approximated is only implicitly represented across the different servers. For example, each of $s$ servers
may have an $n \times d$ matrix $A^t$, and we may be interested in computing a low rank approximation to $A = f(\sum_{t=1}^s A^t)$, where $f$ is a function which is applied entrywise to the matrix $\sum_{t=1}^s A^t$. We show for a wide class of functions $f$ it is possible to efficiently compute a $d \times d$ rank-$k$ projection matrix $P$ for which $\|A - AP\|_F^2 \leq \|A - [A]_k\|_F^2 + \varepsilon \|A\|_F^2$, where $AP$ denotes the projection of $A$ onto the row span of $P$, and $[A]_k$ denotes the best rank-$k$ approximation to $A$ given by the singular value decomposition. The communication cost of our protocols is $d \cdot (sk/\varepsilon)^{O(1)}$, and they succeed with high probability. Our framework allows us to efficiently compute a low rank approximation to an entry-wise softmax, to a Gaussian kernel expansion, and to $M$-Estimators applied entrywise (i.e., forms of robust low rank approximation). We also show that our additive error approximation is best possible, in the sense that any protocol achieving relative error for these problems requires significantly more communication. Finally, we experimentally validate our algorithms on real datasets.
\end{abstract}


%
\IEEEpeerreviewmaketitle

\section{Introduction}
In many situations the input data to large-scale data mining, pattern recognition, and information retrieval tasks is partitioned across multiple servers. This has motivated the distributed model as a popular research model for computing on such data. Communication is often a major bottleneck, and minimizing communication cost is crucial to the success of some protocols. Principal Component Analysis (PCA) is a useful tool for analyzing large amounts of distributed data. The goal of PCA is to find a low dimensional subspace which captures the variance of a set of data as much as possible. Furthermore, it can be used in various feature extraction tasks such as \cite{dang2014discriminative,zhao2010efficient}, and serve as a preprocessing step in other dimensionality reduction methods such as Linear Discriminant Analysis (LDA) \cite{duda2012pattern,cai2008training}.

PCA in the distributed model has been studied in a number of previous works, including \cite{blsw15,bw15,kannan2014principal, feldman2013turning, liang2014improved}. Several communication-efficient algorithms for approximate PCA were provided by these works. In the setting of \cite{feldman2013turning, liang2014improved}, the information of each data point is completely held by a unique server. In \cite{kannan2014principal} a stronger model called the ``arbitrary partition model'' is studied: each point is retrieved by applying a linear operation across the data on each server. Despite this stronger model, which can be used in several different scenarios, it is still not powerful enough to deal with certain cases. For instance, if each data point is partitioned across several servers, and one wants to apply a non-linear operation to the points, it does not apply. For example, the target may be to analyze important components of Gaussian random Fourier features~\cite{rahimi2007random} of data. Another example could be that each person is characterized by a set of health indicators. The value of an indicator of a person may be different across records distributed in each hospital for that person. Because the probability that a person has a problem associated with a health issue increases when a hospital detects the problem, the real value of an indicator should be almost the maximum among all records. This motivates taking the maximum value of an entry shared across the servers, rather than a sum. These examples cannot be captured by any previous model. Therefore, we focus on a stronger model and present several results in it.

\textbf{The model} \textit{(generalized partition model)}\textbf{.}\ \ \ In the generalized partition model, there are $s$ servers labeled $1$ to $s$. Server $t$ has a local matrix $A^t\in\mathbb{R}^{n\times d}$, $(n\gg d)$, of which each row is called a data point. Each of $s$ servers can communicate with server $1$, which we call the Central Processor (CP). Such a model simulates arbitrary point-to-point communication up to a multiplicative factor of $2$ in the number of messages and an additive factor of $\log_2 s$ per message. Indeed, if server $i$ would like to send a message to server $j$, it can send the message to server $1$ instead, together with the identity $j$, and server $1$ can forward this message to server $j$. The global data matrix $A\in\mathbb{R}^{n\times d}$ can be computed given $\{A^t\}_{t=1}^s$. That is, to compute the $(i,j)^{th}$ entry, $A_{i,j}=f(\sum_{t=1}^{s}A^t_{i,j})$, where $f:\mathbb{R}\rightarrow\mathbb{R}$ is a specific function known to all servers. Local computation in polynomial time and linear space is allowed.

\textbf{Approximate PCA and Low-rank matrix approximation.}\ \ \ Given $A\in\mathbb{R}^{n\times d},k\in\mathbb{N}_{+},\varepsilon\in\mathbb{R}_{+}$, a low-rank approximation to $A$ is $AP$, where $P$ is a $d\times d$ projection matrix with rank at most $k$ satisfying
$$||A-AP||_F^2\leq(1+\varepsilon)\min_{X:rank(X)\leq k} ||A-X||_F^2$$
or
$$||A-AP||_F^2\leq\min_{X:rank(X)\leq k} ||A-X||_F^2+\varepsilon||A||_F^2$$
where the Frobenius norm $||A||_F^2$ is defined as $\sum_{i,j}A_{i,j}^2$. If $P$ has the former property, we say $AP$ is a low-rank approximation to $A$ with relative error. Otherwise, $AP$ is an approximation with additive error. Here, $P$ projects rows of $A$ onto a low-dimensional space. Thus, it satisfies the form of approximate PCA. The target is to compute such a $P$.

In section \ref{applications}, we will see that the previous Gaussian random Fourier features and hospital examples can be easily captured by our model.

\textbf{Our contributions.}\ \ \ Our results can be roughly divided into two parts: 1. An algorithmic framework for additive error approximate PCA for general $f(\cdot)$ and several applications. 2. Lower bounds for relative error approximate PCA for some classes of $f(\cdot)$. Our lower bounds thus motivate the error guarantee of our upper bounds, as they show that achieving relative error cannot be done with low communication.

\textit{Algorithmic framework:} For a specific function $f(\cdot)$, suppose there is a distributed sampler which can sample rows from the global data matrix $A$ with probability proportional to the square of their $\ell_2$ norm. Let $Q_j$ be the probability that row $j$ is chosen. If the sampler can precisely report $Q_j$ for a sampled row $j$, a sampling based algorithm is implicit in the work of Frieze, Kannan, and Vempala \cite{frieze2004fast}: the servers together run the sampler to sample a sufficient number of rows. Each server sends its own data corresponding to the sampled rows to server $1$. Server $1$ computes the sampled rows of $A$, scales them using the set \{$Q_j$ $|$ row $j$ is sampled\} and does PCA on the scaled rows. This provides a low-rank approximation to $A$ with additive error. Unfortuately, in some cases, it is not easy to design an efficient sampler which can report the sampling probabilities without some error. Nevertheless, we prove that if the sampler can report the sampling probabilities approximately, then the protocol still works. Suppose the total communication of running the sampler requires $\mathcal{C}$ words. Then the communication of our protocol is $O(sk^2d/\varepsilon^2+\mathcal{C})$ words.

\textit{Applications:} One only needs to give a proper sampler to apply the framework for a specific function $f$.
The softmax (generalized mean) function is important in multiple instance learning (MIL) \cite{babenko2008simultaneous} and some feature extraction tasks such as \cite{boureau2010theoretical}. Combining this with the sampling method in \cite{jowhari2011tight,monemizadeh20101}, a sampler for the softmax (GM) function is provided. Here the idea is for each server to locally raise (the absolute value of) each entry of its matrix to the $p$-th power for a large value of $p > 1$, and then to apply the $\ell_{2/p}$-sampling technique of \cite{jowhari2011tight,monemizadeh20101} on the sum (across the servers) of the resulting matrices. This approximates sampling from the matrix which is the entry-wise maximum of the matrices across the servers, and seems to be a new application of $\ell_{2/p}$-sampling algorithms.

Several works \cite{rahimi2007random,le2013fastfood} provide approximate Gaussian RBF kernel expansions. Because Gaussian kernel expansions have the same length, simple uniform sampling works in these scenarios.

For some $\psi$-functions of M-estimators \cite{zhang1997parameter} such as $L_1-L_2$, ``fair'', and Huber functions, we develop a new sampler which may also be of independent interest. This is related to the fact that such functions have at most quadratic growth, and we can generalize the samplers in \cite{jowhari2011tight,monemizadeh20101} from $p$-th powers to more general functions of at most quadratic growth, similar to the generalization of Braverman and Ostrovsky \cite{bo10a} to the work of Indyk and Woodruff \cite{indyk2005optimal} in the context of norm estimation.

\textit{Lower bounds:} We also obtain several communication lower bounds for computing relative error approximate PCA. These results show hardness in designing relative error algorithms. When $f(x)\equiv\Omega(x^p)$ $(p>1)$, the lower bound is $\tilde{\Omega}((1+\varepsilon)^{-\frac{2}{p}}n^{1-\frac{1}{p}}d^{1-\frac{4}{p}})$ bits. Since $n$ could be very large, this result implies hardness when $f(x)$ grows quickly. When $f(x)=x^p($or $|x|^p)$ $(p\not=0)$, $\Omega(1/\varepsilon^2)$ bits of communication are needed. The result also improves an $\tilde{\Omega}(skd)$ lower bound shown in \cite{kannan2014principal} to $\tilde{\Omega}(\max(skd,1/\varepsilon^2))$. When $f(\cdot)$ is $\max(\cdot)$, we show an $\tilde{\Omega}(nd)$ bit lower bound which motivates us using additive error algorithms as well as a softmax (GM) function.

\textbf{Related work.}\ \ \ Sampling-based additive error low-rank approximation algorithms were developed by \cite{frieze2004fast}. Those algorithms cannot be implemented in a distributed setting directly as they assume that the sampler can report probabilities perfectly, which is sometimes hard in a distributed model. In the row partition model, relative error distributed algorithms are provided by \cite{feldman2013turning,liang2014improved}, and can also be achieved by \cite{ghashami2014relative}. Recently, \cite{kannan2014principal} provides a relative error algorithm in the linear, aforementioned arbitrary partition model. We stress that none of the algorithms above can be applied to our
generalized partition model.

\section{Preliminaries}
We define $[n]$ to be the set $\{1,...,n\}$. Similarly, $[n]-1$ defines the set $\{0,...,n-1\}$.  $A_i$, $A_{:,j}$ and $A_{i,j}$ denotes the $i^{th}$ row, the $j^{th}$ column and the $(i,j)^{th}$ entry of $A$ respectively. $A^t$ is the data matrix held by server $t$. $|\cdot|_p$, $||\cdot||_F$ and $||\cdot||$ means $p$-norm, Frobenius norm and Spectral norm. $[A]_k$ represents the best rank-$k$ approximation to $A$. Unless otherwise specified, all the vectors are column vectors.

If the row space of $A$ is orthogonal to the row space of $B$, $||A||_F^2+||B||_F^2=||A+B||_F^2$. This is the matrix Pythagorean theorem. Thus $||A-AP||_F^2=||A||_F^2-||AP||_F^2$ will be held for any projection matrix $P$. The goal of rank-$k$ approximation to $A$ can be also interpreted as finding a rank-$k$ projection matrix $P$ which maximizes $||AP||_F^2$.

For a vector $v\in \mathbb{R}^m$ and a set of coordinates $\mathcal{S}$, we define $v(\mathcal{S})$ as a vector in $\mathbb{R}^m$ satisfying:
$$\left\{\begin{array}{ll}v(\mathcal{S})_j=v_j&j\in\mathcal{S}\\ v(\mathcal{S})_j=0&j\not\in\mathcal{S}\end{array}\right.$$

\section{Sampling based algorithm}

We start by reviewing several results shown in \cite{frieze2004fast}. Suppose $A\in\mathbb{R}^{n\times d}$ and there is a sampler $\mathfrak{s}$ which samples row $i$ of $A$ with probability $Q_i$ satisfying $Q_i\geq c|A_i|_2^2/||A||_F^2$ for a constant $c\in (0,1]$. Let $\mathfrak{s}$ independently sample $r$ times and the set of samples be $\{A_{i_1},...,A_{i_r}\}$. We construct a matrix $B\in\mathbb{R}^{r\times d}$ such that $\forall i'\in[r],B_{i'}=A_{i_{i'}}/\sqrt{rQ_{i_{i'}}}$. In the following, we will show that we only need to compute the best low rank approximation to $B$ to acquire a good approximation to $A$.

As shown in \cite{frieze2004fast}, $\forall \theta>0$, the condition $||A^TA-B^TB||_F\leq \theta||A||_F^2$ will be violated with probability at most $O(1/(\theta^2r))$. Thus, if the number of samples $r$ is sufficiently large, $||A^TA-B^TB||_F$ is small with high probability.

\begin{Lemma} \label{lemma1}
If $||A^TA-B^TB||_F\leq \theta||A||_F^2$, then for all projection matrices $P'$ satisfying rank$(P')\leq k$,
$$|||AP'||_F^2-||BP'||_F^2|\leq \theta k||A||_F^2.$$
\end{Lemma}

\begin{proof}
Without loss of generality, we suppose rank$(P')=k$. Let $\{U_{:,1},...,U_{:,k}\}$ be an orthogonal basis of the row space of $P'$. We have
\begin{align*}
&\left|||AP'||_F^2-||BP'||_F^2\right|\\
&=\left|||AUU^T||_F^2-||BUU^T||_F^2\right|
=\left|||AU||_F^2-||BU||_F^2\right|\\
&=\left|\sum_{j=1}^k (U_{:,j})^TA^TAU_{:,j}-\sum_{j=1}^k (U_{:,j})^TB^TBU_{:,j} \right|\\
&\leq \sum_{j=1}^k \left|(U_{:,j})^T(A^TA-B^TB)U_{:,j}\right|\\
&\leq \sum_{j=1}^k \left|(U_{:,j})^T\right|_2||A^TA-B^TB||\left|U_{:,j}\right|_2\leq k\theta||A||_F^2
\end{align*}
$||A^TA-B^TB||\leq||A^TA-B^TB||_F\leq \theta||A||_F^2$ implies the last inequality.
\end{proof}

\begin{Lemma} \label{lemma2}
If for all $P'$ satisfying rank$(P')=k$, $|||AP'||_F^2-||BP'||_F^2|\leq \varepsilon||A||_F^2$, then the rank-$k$ projection matrix $P$ which satisfies $BP=[B]_k$ also provides a good rank-$k$ approximation to $A$:
$$||A-AP||_F^2\leq ||A-[A]_k||_F^2+2\varepsilon||A||_F^2$$
\end{Lemma}

\begin{proof}
Suppose $P^*$ provides the best rank-$k$ approximation to $A$. We have,
\begin{align*}
&||AP||_F^2\geq||BP||_F^2-\varepsilon||A||_F^2\geq||BP^*||_F^2-\varepsilon||A||_F^2\\
&\geq||AP^*||_F^2-2\varepsilon||A||_F^2=||[A]_k||_F^2-2\varepsilon||A||_F^2
\end{align*}
\end{proof}

 Thus, combining with Lemma \ref{lemma1} and Lemma \ref{lemma2}, if matrix $B$ has $r=\Theta(k^2/\varepsilon^2)$ rows, the projection matrix $P$ which provides $[B]_k$ also provides an $O(\varepsilon)$ additive error rank-$k$ approximation to $A$ with constant probability. So, we can just compute the projection matrix which provides the best rank-$k$ approximation to $B$.

\section{Framework for distributed PCA}\label{framework}
Our protocol implements the previous sampling based algorithm in a distributed setting. Server $t\in[s]$ holds a local matrix $A^t\in\mathbb{R}^{n\times d}$ and the $(i,j)^{th}$ entry of the global data matrix $A$ is computed by $A_{i,j}=f(\sum_{t=1}^{s}A^t_{i,j})$. A framework for computing a projection matrix $P$ such that $AP$ is a low-rank approximation to $A$ is presented in Algorithm \ref{PCA}. Here, $\mathfrak{s}$ is the same as that defined in the previous section but in the distributed model. Notice that $\mathfrak{s}$ may be different when $f(.)$ is different.
\begin{algorithm}
\caption{Compute $P$}\label{PCA}
  \begin{algorithmic}[1]
  \State \textbf{Input:}  $\{A^t\in \mathbb{R}^{n\times d}\}_{t=1}^s$; $k\in[d]$; $\varepsilon>0$;
  \State \textbf{Output:} rank-$k$ projection matrix $P$.
  \State Set parameters $r=\Theta(\frac{k^2}{\varepsilon^2})$, $\gamma=O(\sqrt{\frac1{r}})$.
  \For {$i':=1\rightarrow r$}
  \State $\mathfrak{s}$ samples row $i_{i'}$ from $A$ and reports $\hat{Q}_{i_{i'}}$ satisfying $(1-\gamma) Q_{i_{i'}}\leq\hat{Q}_{i_{i'}}\leq (1+\gamma)Q_{i_{i'}}$
  \EndFor
  \State $\forall i'\in [r]$, server $t$ sends $A^t_{i_{i'}}$ to server $1$ and server $1$ computes $B\in\mathbb{R}^{r\times d}$ satisfying $B_{i'}=\frac{A_{i_{i'}}}{\sqrt{r\hat{Q}_{i_{i'}}}}$
  \State Server $1$ computes the top $k$ singular vectors $V_{:,1},...,V_{:,k}$ in the row space of $B$ and outputs $P=VV^T$
  \end{algorithmic}
\end{algorithm}

In Algorithm \ref{PCA}, each row $i'$ of $B$ is scaled via $\hat{Q}_{i_{i'}}$ but not $Q_{i_{i'}}$. However this is not an issue.
\begin{Lemma}\label{lemma3}
For $A$, $B$ in Algorithm \ref{PCA}, $||A^TA-B^TB||_F\leq O(\varepsilon/k)||A||_F^2$ holds with constant probability.
\end{Lemma}
\begin{proof}
We set $r=\lceil\frac{1440k^2}{\varepsilon^2c}\rceil$ where $c$ is a constant satisfying the condition: $\forall i\in[n],Q_i\geq c|A_i|_2^2/||A||_F^2$. Assume $\gamma\leq0.5$. Then we have $\forall i\in[r],1-2\gamma\leq Q_{j_i}/\hat{Q}_{j_i}\leq1+2\gamma$.
\begin{align*}
& \left|E((B^TB)_{i,j})-(A^TA)_{i,j}\right|\\
&=\left|\sum_{t=1}^r\sum_{m=1}^nQ_m\frac{A_{m,i}A_{m,j}}{r\hat{Q}_m}-(A^TA)_{i,j}\right|\\
&=\left|\sum_{m=1}^nA_{m,i}A_{m,j}(\frac{Q_m}{\hat{Q}_m}-1)\right|\\
&\leq 2\gamma\sum_{m=1}^n\left|A_{m,i}A_{m,j}\right|\\
&\leq 2\gamma |A_{:,i}|_2|A_{:,j}|_2
\end{align*}
Due to $x^2 \leq 2(x-y)^2 + 2y^2$ for any $x,y\in\mathbb{R}$,
\begin{align*}
& E(((B^TB)_{i,j}-(A^TA)_{i,j})^2)\\
&\leq 2E(((B^TB)_{i,j}-E(B^TB)_{i,j})^2)+8\gamma^2|A_{:,i}|^2_2|A_{:,j}|^2_2 \\
&\leq 2\sum_{t=1}^r E((B_{t,i}B_{t,j})^2)+8\gamma^2|A_{:,i}|^2_2|A_{:,j}|^2_2\\
&\leq 2(1+2\gamma)^2\frac{||A||_F^2}{cr}\sum_{t=1}^n\frac{A_{t,i}^2A_{t,j}^2}{|A_t|_2^2}+8\gamma^2|A_{:,i}|^2_2|A_{:,j}|^2_2
\end{align*}
\begin{align*}
& E(||(B^TB)-(A^TA)||_F^2)\\
&=\sum_{i,j=1}^dE(((B^TB)_{i,j}-(A^TA)_{i,j})^2)\\
&\leq2(1+2\gamma)^2\frac{||A||_F^2}{cr}\sum_{t=1}^n\frac1{|A_t|_2^2}\sum_{i,j=1}^dA_{t,i}^2A_{t,j}^2+\frac{8}{cr}||A||_F^4\\
&\leq \frac{16}{cr}||A||_F^4\leq\frac{\varepsilon^2}{90k^2}||A||_F^4
\end{align*}
Due to Markov's inequality,
$$\Pr(||A^TA-B^TB||_F\geq \frac{\varepsilon}{3k}||A||_F^2)\leq \frac1{10}$$
\end{proof}

Without the sampling stage, the only communication involved is to collect $\Theta(k^2/\varepsilon^2)$ rows from $s$ servers.
\begin{Theorem} \label{upper1}
With constant probability, Algorithm \ref{PCA} outputs $P$ for which $AP$ is an $O(\varepsilon)$ additive error rank-$k$ approximation to $A$. Moreover, if the communication of sampling and reporting $\hat{Q}_{\{i_1,...,i_r\}}$ uses $\mathcal{C}$ words, the communication of the overall protocol needs $O(sk^2d/\varepsilon^2+\mathcal{C})$ words.
\end{Theorem}
In addition, to boost the success probability to $1-\delta$, we can just run Algorithm \ref{PCA} $O(\log(1/\delta))$ times and output the matrix $P$ with maximum $||BP||_F^2$. The communication is the same as before up to an $O(\log(1/\delta))$ factor.

\section{Generalized sampler}
To apply Algorithm \ref{PCA} for a specific $f(\cdot)$, we should implement the distributed sampler $\mathfrak{s}$ which can sample rows of a global data $A$, and row $i$ is sampled with probability approximately proportional to $|A_i|_2^2$. An observation is that $|A_i|_2^2=\sum_{j=1}^d A_{i,j}^2$. Thus, the row sampling task can be converted into an entry sampling task. If an entry is sampled, then we choose the entire row as the sample. In the entry-sampling task, we want to find a sampler which can sample entries of $A$ such that $A_{i,j}$ is chosen with probability proportional to $A_{i,j}^2=f(\sum_{t=1}^s A^t_{i,j})^2$. Another observation is that in Algorithm \ref{PCA}, we only need to guarantee the probability that row $i$ is chosen is not less than $c|A_i|_2^2/||A||_F^2$ for a constant $c$. Thus, if there exists a function $z(x)$ and a constant $c\geq1$ such that $z(x)/c\leq f(x)^2\leq cz(x)$, then the sampler which samples $A_{i,j}$ with probability proportional to $z(A_{i,j})$ can be used in Algorithm \ref{PCA} when computing $f$.

We consider an abstract continuous function $z(x):\mathbb{R}\rightarrow\mathbb{R}$ which satisfies the property $\mathcal{P}$: $\forall x_1,x_2\in\mathbb{R},\text{ if }|x_1|\geq|x_2|,\ x_1^2/z(x_1)\geq x_2^2/z(x_2),z(x_1)\geq z(x_2)$, and $z(0)=0$. Suppose $a^t\in\mathbb{R}^l$ is located on server $t$. Let $a=\sum_{t=1}^s a^t$ and $Z(a)=\sum_{i=1}^l z(a_i)$.
\begin{Theorem}\label{theoremsampler}
There is a protocol which outputs $i\in[l]$ with probability $(1\pm\varepsilon)z(a_i)/Z(a)+O(l^{-C})$ and reports a $(1\pm \varepsilon)$ approximation to $Z(a)$ with probability at least $1-O(l^{-C})$, where $C$ is a non-negative constant. The communication is $s\cdot poly(log(l)/\varepsilon)$ bits.
\end{Theorem}
Notice that the additive error of $O(l^{-C})$ can effectively be treated as
relative error, since if
$O(l^{-C})$ is an $\varepsilon$-fraction of $z(a_j)/Z(a)$, then the probability that $j$ is sampled at least once is negligible. Therefore, this sampler can be applied in Algorithm \ref{PCA}. In the following, we will show how to sample coordinates of $a$ with the desired distribution.

\subsection{Overview}
We give an overview of our sampler. We define $S_i$ to be the interval $[(1+\varepsilon)^i,(1+\varepsilon)^{i+1})$. Similar to \cite{monemizadeh20101}, coordinates of $a$ are conceptually divided into several classes
$$S_i(a)=\{j\ |\ z(a_j)\in S_i\}$$
Here, $i\in\mathbb{Z}$. Although the number of classes is infinite, there are at most $l$ non-empty classes which we care about. To sample a coordinate, our procedure can be roughly divided into two steps:
\begin{enumerate}
\item Sample a class $S_i(a)$ with probability approximately proportional to $\sum_{j\in S_i(a)}z(a_j)/Z(a)$.
\item Uniformly output a coordinate in $S_i(a)$.
\end{enumerate}
For convenience, we call $\sum_{j\in S_i(a)}z(a_j)$ the contribution of $S_i(a)$. To do the first step, we want to estimate $Z(a)$ and the contribution of each class. We can achieve this by estimating the size of each class. Intuitively, if the contribution of $S_i(a)$ is ``non-negligible'', there are two possible situations: 1. $\forall j\in S_i(a)$, $j$ is ``heavy'', $z(a_j)/Z(a)$ is large; 2. the size of $S_i(a)$ is large. For the first case, we just find out all the ``heavy'' coordinates to get the size of $S_i(a)$. For the second case, if we randomly drop some coordinates, there will be some coordinates in $S_i(a)$ survived, and all of those coordinates will be ``heavy'' among survivors. Thus we can find out all the ``heavy'' survivors to estimate the size of $S_i(a)$. The second step is easier, we will show it in \ref{thesampler}.

In \ref{heavyhitters}, we propose a protocol which can pick out all the ``heavy'' coordinates. In \ref{estimationofz}, we demonstrate the estimation of the size of each class and $Z(a)$. In \ref{thesampler}, we complete our sampler.

\subsection{Z-HeavyHitters} \label{heavyhitters}
A streaming algorithm which can find frequent elements whose squared frequency is large compared to the second moment (the ``$F_2$-value'' of a stream) is described in \cite{ccf04}. Because it provides a linear sketch, it can be easily converted into a distributed protocol in our setting. We call the protocol \textbf{HeavyHitters}. The usage of \textbf{HeavyHitters} is as follows. $v^t\in\mathbb{R}^m$ is located in server $t$. Let $v=\sum_{t=1}^s v^t$. With success probability at least $1-\delta$, \textbf{HeavyHitters}$(v, B, \delta)$ reports all the coordinates $j$ such that $v_j^2\geq |v|_2^2/B$. Server $1$ can get reports from \textbf{HeavyHitters}. The communication of this protocol needs $O(sB\cdot \textrm{poly}(\log(m))\log(1/\delta))$ words.

By using \textbf{HeavyHitters}, we develop a protocol called \textbf{Z-HeavyHitters}$(v\in\mathbb{R}^m, B, \delta)$ shown in Algorithm \ref{improvedheavyhitters}.
\begin{algorithm}
\caption{\textbf{Z-HeavyHitters}}\label{improvedheavyhitters}
  \begin{algorithmic}[1]
  \State \textbf{Input:}  $v\in\mathbb{R}^m$; $B>0$; $\delta>0$.
  \State \textbf{Output:} $D$
  \State $\forall t\in[\lceil20\log(1/\delta)\rceil]$, Server $1$ samples $hash_{t}:[m]\rightarrow [\lceil4B^2\rceil]$ from a pairwise independent family of hash functions. Server $1$ broadcasts random seeds. Thus $\forall t$, all servers know $hash_{t}$.
  \State Server $1$ sets $D=\emptyset$.
  \For {$t:=1\rightarrow \lceil20\log(1/\delta)\rceil$}
    \For {$e:=1\rightarrow \lceil4B^2\rceil$}
    \State $H_{t,e}:=\{j\ |\ hash_{t}(j)=e\}$
    \State $D:=D\cup$\textbf{HeavyHitters}$(v(H_{t,e}),B,1/(16B^2))$
    \EndFor
  \EndFor
  \State report $D$
  \end{algorithmic}
\end{algorithm}

\begin{Lemma}\label{lemmaimproveheavy}
 With probability at least $1-\delta$, $\forall j\in[m]$ satisfying $z(v_j)\geq Z(v)/B$, $j$ is reported by \textbf{Z-HeavyHitters}$(v, B, \delta)$. The communication cost is $O(s\cdot \textrm{poly}(B\log(m))\log(1/\delta))$.
\end{Lemma}
The intuition is that any two coordinates which are both heavy in $Z(v)$ do not collide in the same bucket with constant probability. Due to property $\mathcal{P}$ of $z(\cdot)$, the ``heavy'' coordinate will be reported by \textbf{HeavyHitters}.
\begin{proof}
Because $hash_t$ is from a pairwise independent family, for fixed $t\in[\lceil20\log(1/\delta)\rceil],j,j'\in[m],j\not=j'$, the probability that $hash_t(j)=hash_t(j')$ is $(4B^2)^{-1}$. According to the union bound, for a fixed $t$, the probability that any two different coordinates $j,j'\in[m]$ with $z(v_j),z(v_{j'})\geq Z(v)/B$ do not collide is at least $3/4$.

\textbf{Claim:} For fixed $t,e$, suppose $H_{t,e}$ has at most one coordinate $j$ such that $z(v_j)\geq Z(v)/B$. $j$ is reported by \textbf{HeavyHitters}$(v(H_{t,e}),B,1/(16B^2))$ with probability at least $1-1/(16B^2)$.

If $j$ is the unique coordinate in the set $H_{t,e}$ such that $z(v_j)\geq Z(v)/B$, then we have $(v_j)^2/z(v_j)\geq|v(H_{t,e})|_2^2/Z(v(H_{t,e}))$ due to the property $\mathcal{P}$ of function $z(\cdot)$. Because $Z(v)\geq Z(v(H_{t,e}))$, we have $(v_j)^2/|v|_2^2\geq z(v_j)/Z(v) \geq 1/B$. When \textbf{HeavyHitters}$(v(H_{t,e}),B,1/(16B^2))$ succeeds, $j$ is reported. Thus the \textbf{claim} is true. Due to the union bound, for a fixed $t$, the probability that all invocations of \textbf{HeavyHitters} succeed is at least $3/4$.
Therefore, for a fixed round $t$, the probability that all the coordinates $j$ such that $z(v_j)\geq Z(v)/B$ are added in to $D$ is at least $1/2$. Due to a Chernoff bound, the probability that there is a round $t\in[\lceil20\log(1/\delta)\rceil]$ such that all the coordinates $j$ with $z(v_j)\geq Z(v)/B$ are added into $D$, is at least $1-\delta$.
\end{proof}

\subsection{Estimation of $|S_i(a)|$ and $Z(a)$} \label{estimationofz}
  Set $T=\lceil C\log(l)/\varepsilon+1\rceil$. Here, $C$ is a constant defined in Theorem \ref{theoremsampler}. If $|S_i(a)|(1+\varepsilon)^i\geq \varepsilon Z(a)/(40T)$, we say $S_i(a)$ \textit{contributes}. Define $S_i(a)$ to be \textit{considering} when $(1+\varepsilon)^{i+1}\in[Z(a)/l^C,(1+\varepsilon)Z(a)]$. Due to $T/\varepsilon=o(l)$, for a sufficiently large constant $C$, a contributing $S_i(a)$ must be considering. Define $Z^{NC}(a)$ to be the sum of all $z(a_j)$ satisfying that $j$ belongs to a non-contributed $S_i(a)$. $Z^{NC}(a)$ is bounded as:
$$Z^{NC}(a)<l\cdot Z(a)/l^C+\sum_{\text{consid.,noncontr. }S_i(a)}|S_i(a)|(1+\varepsilon)^{i+1}$$
 Because $l^{1-C}<\varepsilon/2$ and the number of considering classes is at most $T$, $Z^{NC}(a)<\varepsilon Z(a)$. This implies that if $|S_i(a)|$ can be estimated for all contributing $S_i(a)$, we can get a $(1+\Theta(\varepsilon))$ approximation to $Z(a)$.
\begin{algorithm}
\caption{\textbf{Z-estimator}}\label{estimator}
  \begin{algorithmic}[1]
  \State \textbf{Input:}  $\{a^t\in\mathbb{R}^l\}_{t=1}^s$; $\varepsilon>0$.
  \State \textbf{Output:} $\hat{Z},\hat{s}_{-\infty},...,\hat{s}_{+\infty}$,$List$,$g(\cdot)$
  \State $B=40\varepsilon^{-4}T^3\log(l)$, $W=(5120C^2T^2\varepsilon^{-3}\log(l))^2$
  \State Initialize: $\hat{Z}=0,\hat{s}_{-\infty}=0,...,\hat{s}_{+\infty}=0$
  \State $D:=$\textbf{Z-HeavyHitters}$(a,B,l^{-20C})$
  \State $\forall j\in D$, server $1$ asks $a^t_j$ for $t\in[s]$ to compute $a_j$ and makes $\hat{s}_i=|S_i(a)\cap D|$.
  \State Server $1$ samples $g:[\lceil C\varepsilon^{-1}l \rceil]\rightarrow[\lceil C\varepsilon^{-1}l \rceil]$ from a $(20C\log(\varepsilon^{-1}l))$-wise independent family of hash functions. $\forall j\in[\lceil \log(C\varepsilon^{-1}l)] \rceil,e\in[\lceil C\log(l)\rceil]$, $h_{j,e}:[l]\rightarrow[W]$ is independently sampled from a set of pairwise independent hash functions. Server $1$ broadcasts random seeds.
  \For {$j\in[\lceil \log(C\varepsilon^{-1}l) \rceil],e\in[\lceil C\log(l)\rceil],w\in[W]$}
  \State $\mathcal{S}_j:=\{i\ |\ i\in[l],g(i)\leq 2^{-j}\lceil C\varepsilon^{-1}l\rceil\}$
  \State $\mathcal{S}_{j,e,w}:=\{i\ |\ i\in\mathcal{S}_j,h_{j,e}(i)=w\}, D_{j,e,w}:=\text{\bf{Z-HeavyHitters}}(a(\mathcal{S}_{j,e,w}),B,l^{-20C})$
  \State $D_j:=\bigcup_{e,w}D_{j,e,w}$. Server $1$ communicates with other servers to compute $a_p$ for all $p\in D_j$.
  \State $\forall i\in\mathbb{Z}, 16C^2\varepsilon^{-2}\log(l)>|S_i(a)\cap D_j|\geq 4C^2\varepsilon^{-2}\log(l)$, $\hat{s}_i:=\max(\hat{s}_i,2^j|S_i(a)\cap D_j|)$
  \EndFor
  \State Report $\hat{Z}:=\sum_i\hat{s}_i(1+\varepsilon)^i$, $\hat{s}_{-\infty},...,\hat{s}_{+\infty}$, $List:=D\cup\bigcup_j D_j$ and $g(\cdot)$.
  \end{algorithmic}
\end{algorithm}

In Algorithm \ref{estimator}, we present a protocol \textbf{Z-estimator} which outputs $\hat{Z}$ as an estimation of $Z(a)$ and $\hat{s}_i$ as an estimation of $|S_i(a)|$. The output $List$ and $g(\cdot)$ will be used in Algorithm \ref{finalsampler} to output the final coordinate. For convenience to analyze the algorithm, we define $U(a)=\bigcup_{nonconsidering} S_i(a)$.
\begin{Lemma}[Proposition 3.1 in \cite{monemizadeh20101}]\label{events}
With probability at least $1-l^{-\Theta(C)}$, the following event $\mathcal{E}$ happens:
\begin{enumerate}
\item $\forall R\in\{S_i(a)\ |\ S_i(a)\text{ is considering}\}\cup\{U(a)\}, \forall j\in[\lceil \log(C\varepsilon^{-1}l) \rceil]$, $|R\cap \mathcal{S}_j|\leq 2\cdot2^{-j}|S_i(a)|+C^2\log(l)$.
\item If $E(R\cap\mathcal{S}_j)\geq C^2\varepsilon^{-2}\log(l)$, $2^j|R\cap\mathcal{S}_j|=(1\pm\varepsilon)|R|$.
\item All the invocations of \textbf{Z-HeavyHitters} succeed.
\end{enumerate}
\end{Lemma}
\begin{Lemma} \label{contri}
When $\mathcal{E}$ happens, if $S_i(a)$ contributes, with probability $1-l^{-\Theta(C)}$ the following happens:
\begin{enumerate}
 \item $\hat{s}_i\geq(1-\varepsilon)|S_i(a)|$.
 \item $S_i(a)\subseteq D$ or $\exists j\in[\lceil \log(C\varepsilon^{-1}l) \rceil],S_i(a)\cap\mathcal{S}_j=S_i(a)\cap D_j\not=\emptyset$.
\end{enumerate}
\end{Lemma}
The intuition of Lemma \ref{contri} is that if $S_i(a)$ contributes, there are two cases: if $(1+\varepsilon)^i$ is very large, it will be reported in $D$. Otherwise, there will be a $j^*$ such that $\sum_{i'\geq i}|\mathcal{S}_{j^*}\cap S_{i'}(a)|$ is bounded. Therefore, for fixed $e$, all the elements which belong to $S_{i'\geq i}(a)$ can be hashed into different buckets. Meanwhile, $(1+\varepsilon)^i$ is heavy in $\sum_{i'<i}|\mathcal{S}_{j^*}\cap S_{i'}(a)|(1+\varepsilon)^{i'+1}$. Thus, all elements in $\mathcal{S}_{j^*}\cap S_{i}(a)$ will be reported by $D_{j^*}$. $|S_{i}(a)|$ can thus be estimated.
\begin{proof}
Assume $\mathcal{E}$ happens. Suppose $|S_i(a)|$ contributes. If $|S_i(a)|<16C^2\varepsilon^{-2}\log(l)$, due to $|S_i(a)|(1+\varepsilon)^i\geq \varepsilon Z(a)/(40T)$,
$$(1+\varepsilon)^i\geq Z(a)/(640C^2\varepsilon^{-3}T\log(l))$$
Because $B=40\varepsilon^{-4}T^3\log(l)>640C^2\varepsilon^{-3}T\log(l)$, all the elements in $S_i(a)$ will be in $D$ when line 5 is executed. $$\hat{s}_i=|S_i(a)\cap D|=|S_i(a)|$$
If $|S_i(a)|\geq16C^2\varepsilon^{-2}\log(l)$, there exists a unique $j^*$ such that $8C^2\varepsilon^{-2}\log(l)\leq2^{-j^*}|S_i(a)|<16C^2\varepsilon^{-2}\log(l)$. For $i'\geq i$,
$$|S_i(a)|(1+\varepsilon)^i\geq\varepsilon Z(a)/(40T)\geq\varepsilon |S_{i'}(a)|(1+\varepsilon)^{i'}/(40T)$$
So,
\begin{align*}
&2^{-j^*}|S_{i'}(a)|\\
&\leq 40T\varepsilon^{-1}2^{-j^*}|S_i(a)|\\
&\leq 40T\varepsilon^{-1}\cdot16C^2\varepsilon^{-2}\log(l)\\
&=640C^2\varepsilon^{-3}T\log(l)
\end{align*}
Due to $\mathcal{E}$, $|S_{i'}(a)\cap\mathcal{S}_{j^*}|\leq 1280C^2\varepsilon^{-3}T\log(l)+C^2\log(l)\leq2560C^2\varepsilon^{-3}T\log(l)$. Because the number of considering classes is bounded by $T$,
$$\sum_{i'\geq i}|S_{i'}(a)\cap\mathcal{S}_{j^*}|\leq 2560C^2\varepsilon^{-3}T^2\log(l).$$
For a fixed $e$, define the event $\mathcal{F}:$ $\forall p,q\in \bigcup_{i'\geq i} S_i(a),p\not=q$, $h_{j^*,e}(p)\not=h_{j^*,e}(q)$. Because $h_{j^*,e}$ is pairwise independent, $Pr(\neg\mathcal{F})$ is bounded by $(2560C^2\varepsilon^{-3}T^2\log(l))^2/W=1/4$ by a union bound.

For $i'<i$, because $S_i(a)$ contributes, we have:
$$(1+\varepsilon)^i|S_i(a)|\geq \varepsilon(1+\varepsilon)^{i'}|S_{i'}(a)|/(40T)$$
Then, $$|S_{i'}(a)|\leq 40T\varepsilon^{-1}|S_i(a)|(1+\varepsilon)^{i-i'}$$
Thus,
\begin{align*}
&|S_{i'}(a)\cap\mathcal{S}_{j^*}|(1+\varepsilon)^{i'} \\
& \leq 2\cdot40T\varepsilon^{-1}2^{-j^*}|S_i(a)|(1+\varepsilon)^{i}+C^2\log(l)(1+\varepsilon)^{i'}\\
& \leq 4\cdot40T\varepsilon^{-1}\cdot 16C^2\varepsilon^{-2}\log(l)(1+\varepsilon)^{i}\\
& = 2560C^2\varepsilon^{-3}T\log(l) (1+\varepsilon)^i
\end{align*}
Therefore, conditioned on $\mathcal{E},\mathcal{F}$, for fixed $e$, $\forall p\in S_i(a)\cap\mathcal{S}_{j^*,e,w}$ we have
\begin{align*}
& Z(a(\mathcal{S}_{j^*,e,w}))\\
&\leq \sum_{i'<i} 2560C^2\varepsilon^{-3}T\log(l)(1+\varepsilon)^i+z(a_p)\\
&\leq 5120C^2\varepsilon^{-3}T^2\log(l)(1+\varepsilon)^{i}
\end{align*}
Because $z(a_p)\geq Z(a(\mathcal{S}_{j^*,e,w}))/B$, $S_i(a)\cap\mathcal{S}_{j^*}=S_i(a)\cap \bigcup_w D_{j^*,e,w}$.

Since $e$ is iterated from $1$ to $\lceil C\log(l)\rceil$ and $D_{j^*}=\bigcup_{e,w}D_{j^*,e,w}$, with probability $1-l^{\Theta(-C)}$, $S_i(a)\cap\mathcal{S}_{j^*}=S_i(a)\cap D_{j^*}$. Notice that $(1-\varepsilon)8C^2\varepsilon^{-2}\log(l)\geq4C^2\varepsilon^{-2}\log(l)$. Thus, when $j=j^*$, $\hat{s}_{i}$ will be set to be $(1\pm\varepsilon)|S_i(a)|$ with probability $1-l^{\Theta(-C)}$.
\end{proof}

\begin{Theorem}\label{Zestimator}
With probability at least $1-l^{-\Theta(C)}$, Algorithm \ref{estimator} outputs $\hat{Z}=(1\pm\Theta(\varepsilon))Z(a)$ and $\forall i\in\mathbb{Z},\hat{s}_i\leq (1+\varepsilon)|S_i(a)|$.
\end{Theorem}
\begin{proof}
Suppose $\mathcal{E}$ happens. If $\hat{s}_i$ is assigned by line 6, $\hat{s}_i=|D\cap S_i(a)|\leq|S_i(a)|$. Otherwise, $\hat{s}_i$ is assigned by line 17. Due to Lemma \ref{events}, $2^j|D_j\cap S_i(a)|\leq2^j|\mathcal{S}_j\cap S_i(a)|=(1\pm\varepsilon)|S_i(a)|$. Thus, $\forall i\in\mathbb{Z},\hat{s}_i\leq (1+\varepsilon)|S_i(a)|$. Because of Lemma \ref{contri} and the bound of $Z^{NC}(a)$, $\hat{Z}=(1\pm\Theta(\varepsilon))Z(a)$.
\end{proof}
Actually, Algorithm \ref{estimator} can be implemented with two rounds of communication. In the first round, all the servers together compute $List$. In the second round, server $1$ checks each element in $List$ and estimates $\hat{s}_i$.

\subsection{The sampler} \label{thesampler}
Since $\hat{Z}(a)$ is a $(1\pm\varepsilon)$ approximation to $Z(a)$, the coordinate injection technique of \cite{monemizadeh20101} can be applied. Define considering $S_i(a)$ to be \textit{growing} when $(1+\varepsilon)^i\leq \hat{Z}(a)/(5\varepsilon^{-4}T^3\log(l))$. If $S_i(a)$ is growing, then server $1$ appends $\lceil \varepsilon \hat{Z}(a)/(5T(1+\varepsilon)^i)\rceil$ coordinates with value $z^{-1}((1+\varepsilon)^i)$ to vector $a^1$ and other servers append a consistent number of $0$s to their own vectors. Since $z(\cdot)$ is an increasing function, if $z^{-1}((1+\varepsilon)^i)$ does not exist, $S_i(a)$ must be empty, we can ignore this class. Thus server $t$ can obtain a vector $a'^t$ and global $a'=\sum_{t=1}^s a'^t$. Similar to Lemma 3.2,3.3,3.4 of \cite{monemizadeh20101}, $Z(a)\leq Z(a')\leq (1+\varepsilon)Z(a)$ and $\forall \text{growing }S_i(a)$, $S_i(a')$ is contributed. Furthermore, the dimension of $a'$ is $l'=poly(l)$. We get the final sampler in Algorithm \ref{finalsampler}.
\begin{algorithm}
\caption{\textbf{Z-sampler}}\label{finalsampler}
  \begin{algorithmic}[1]
  \State \textbf{Input:}  $\{a'^t\in\mathbb{R}^{l'}\}_{t=1}^s$; $\varepsilon>0$
  \State \textbf{Output:} coordinate $p$ of $a'$
  \State $\hat{Z},\hat{s}_{-\infty},...,\hat{s}_{+\infty},List,g:=$\textbf{Z-estimator}$(a',\varepsilon)$
  \State Choose $i^*\in\mathbb{Z}$ with probability $\hat{s}_{i^*}(1+\varepsilon)^{i^*}/\hat{Z}$
  \State Choose $p$ satisfying $g(p)=\min_{q\in List\cap S_{i^*}(a')}g(q)$
  \State if $p$ is not an injected coordinate, output $p$. Otherwise, output \textbf{FAIL}.
  \end{algorithmic}
\end{algorithm}
\begin{Lemma} \label{list}
With probability $1-l^{-\Theta(C)}$, the following happens for all considering $S_i(a')$, $\hat{s}_i=(1\pm\varepsilon)|S_i(a')|$. If $p_i$ satisfies $g(p_i)=\min_{q\in S_i(a')}g(q)$, then $p_i$ belongs to $List$.
\end{Lemma}
\begin{proof}
Suppose \textbf{Z-estimator} succeeds. If $S_i(a)$ is growing, then $S_i(a')$ contributes. Thus, due to Lemma \ref{contri} and Theorem \ref{Zestimator}, $\hat{s}_i=(1\pm\varepsilon)|S_i(a')|$. Point 2 of Lemma \ref{contri} also implies that $p_i\in List$. If $S_i(a)$ is not growing, $(1+\varepsilon)^i>\hat{Z}(a)/(5\varepsilon^{-4}T^3\log(l))>Z(a')/B$. $D\cap S_{i}(a')=S_{i}(a')$. Therefore, $p_i\in List$ and $\hat{s}_i=|S_i(a')|$
\end{proof}
Due to Lemma \ref{list}, in line 5 of Algorithm \ref{finalsampler}, $p$ satisfies that $g(p)=\min_{q\in S_{i^*}(a')}g(q)$.
\begin{Lemma}[Theorem 3.3 in \cite{monemizadeh20101}]\label{uniform} $$Pr(g(p)=\min_{q\in S_{i}(a')}g(q))=(1\pm \Theta(\varepsilon))/|S_{i}(a')|$$
\end{Lemma}
Since $\hat{s}_{i^*}=(1\pm\varepsilon)|S_{i^*}(a')|$, according to Theorem \ref{Zestimator}, Lemma \ref{list} and Lemma \ref{uniform}, \textbf{Z-sampler} samples coordinate $p$ with probability $(1\pm\Theta(\varepsilon))z(a'_p)/Z(a')\pm l^{-\Theta(C)}$. Since injected coordinates contributed at most $\Theta(\varepsilon)$ to $Z(a')$, the probability that \textbf{Z-sampler} outputs \textbf{FAIL} is at most $\Theta(\varepsilon)$. Thus, we can repeat \textbf{Z-sampler} $O(C\log(l))$ times and choose the first non-injected coordinate. The probability that \textbf{Z-sampler} outputs a non-injected coordinate at least once is $1-l^{\Theta(-C)}$. Because $Z(a')=(1\pm\varepsilon)Z(a)$, each coordinate $i$ is sampled with probability $(1\pm\Theta(\varepsilon))z(a_i)/Z(a)\pm l^{-\Theta(C)}$. By adjusting $C$ and $\varepsilon$ by a constant factor, Theorem \ref{theoremsampler}
 is shown.

\section{Applications}\label{applications}
\subsection{Gaussian random fourier features}
Gaussian RBF kernel \cite{shawe2004kernel} on two $d$-dimensional vectors $x,y$ is defined as
$$K(x,y)=\phi(x)^T\phi(y)=e^{-\frac{|x-y|_2^2}2}$$
According to the Fourier transformation of $K(x,y)$,
$$K(x,y)=\int_{\mathbb{R}^d}p(z)e^{iz^T(x-y)}dz$$
where $p(z)=(2\pi)^{-\frac{d}2}e^{-\frac{|z|_2^2}{2}}$. Suppose $z$ is a $d$-dimensional random vector with each entry drawn from $N(0,1)$, As shown in \cite{rahimi2007random},\cite{le2013fastfood}, estimating $E_{z}(e^{iz^Tx}{e^{-iz^Ty}})$ by sampling such a vector $z$ provides a good approximation to $K(x,y)$. According to \cite{rahimi2007random}, if the samples are $z_1,...,z_l$, then $\phi(x)$ can be approximated by $$\sqrt{2}(\cos(z_1^Tx+b_1),...,\cos(z_l^Tx+b_l))^T$$
 where $b_1,...,b_l\in\mathbb{R}$ are i.i.d.samples drawn from a uniform distribution on $[0,2\pi]$.

 In the distributed setting, matrix $M^i\in\mathbb{R}^{n\times m}$ is stored in server $i\in[s]$. The global matrix is computed by $M=\sum_{t=1}^s M^t$. Let each entry of $Z\in\mathbb{R}^{m\times d}$ be an i.i.d. sample drawn from $N(0,1)$ and each entry of $b\in\mathbb{R}^d$ be an i.i.d. sample drawn from uniform distribution on $[0,2\pi]$. An approximated Gaussian RBF kernel expansion of $M_i$ is $A_i$ of which
 $$A_{i,j}=\sqrt{2}\cos((M_iZ)_j+b_j)$$
 For fixed $i,j$, one observation is
 $$E(A_{i,j}^2)=E(2\cos^2((M_iZ)_j+b_j))=1.$$
Due to Hoeffding's inequality, when $d=\Theta(\log(n))$, the probability that $\forall i\in[n], |A_i|_2^2=\Theta(d)$ is high.  Thus server $1$ can obtain $O(k^2/\varepsilon^2)$ rows of $M$ via uniform sampling and generate $Z,b$ with $d=O(\log(n))$ to compute approximate PCA on these random Fourier features of $M$.

\subsection{Softmax (GM)}
The generalized mean function $GM(\cdot)$ with a parameter $p$ of $n$ positive reals $x_1,...,x_n$ is defined as: $$GM(x_1,...,x_n)=\left(\frac1n\sum_{i=1}^nx_i^p\right)^{\frac1p}$$
When $p$ is very large, $GM(\cdot)$ behaves like $\max(\cdot)$. If $p=1$, $GM(\cdot)$ is just $mean(\cdot)$. We discuss the following situation: each server $i\in[s]$ holds a local matrix $M^i\in\mathbb{R}^{n\times d}$. The global data $A$ is formed by
$$A_{i,j}=GM(|M^1_{i,j}|,...,|M^s_{i,j}|)$$
for $p\geq 1$. Since server $t$ can locally compute $A^t$ such that
$$A^t_{i,j}=\frac1s(M^t_{i,j})^p$$
the setting meets the generalized partition model with $f(x)\equiv x^{\frac1p}$. So, we can apply the sampling algorithm in \cite{jowhari2011tight,monemizadeh20101}. Because $f^2(x)$ holds the property $\mathcal{P}$, our generalized sampler also works in this scenario. Furthermore, the communication costs of our algorithm does not depend on $p$. Therefore, if we set $p=\log(nd)$, the word size of $A^t_{i,j}$ is the same of the word size of $M^t_{i,j}$ up to a factor $\log(nd)$. But for an arbitrary constant $c'\in (0,1)$, when $n,d$ are sufficient large, $$GM(|M^1_{i,j}|,...,|M^s_{i,j}|)>c'\max(|M^1_{i,j}|,...,|M^s_{i,j}|)$$
can be held. As shown in the results of lower bounds, computing relative error approximate PCA for $\max(\cdot)$ is very hard.

\subsection{M-estimators}
In applications, it is possible that some parts of the raw data are contaminated by large amount of noises. Because traditional PCA is sensitive to outliers, certain forms of robust PCA have been developed. By applying a proper function $f(\cdot)$, we can find a low rank approximation to a matrix that has had a threshold applied to each of its entries, that is, the matrix $A$ has its $(i,j)^{th}$ entry bounded by a threshold $T$. It is useful if one wants to preserve the magnitude of entries except for those that have been damaged by noise, which one wants to cap off. $\psi$-functions of some M-estimators can be used to achieve this purpose, and parts of them are listed in table \ref{sample-table}. Suppose $\psi(x)$ is one of those functions and the global data $A$ satisfies $A_{i,j}=\psi(\sum_{t=1}^s A^t_{i,j})$. Because $\psi(x)^2$ satisfies the property $\mathcal{P}$, combining with our framework and generalized sampler, it works for computing approximate PCA for such $A$. However, computing relative error approximate PCA to $A$ may be very hard. Our result shows that at least $\tilde{\Omega}(nd)$ bits are needed to get relative error when $f(\cdot)$ is $\psi$-function of Huber.
\begin{table}[h!]
\caption{$\psi$-functions of several M-estimators}
\label{sample-table}
\small{
\begin{center}
\begin{tabular}{ccc}
\multicolumn{1}{c}{ Huber}  &\multicolumn{1}{c}{ $L_1-L_2$} &\multicolumn{1}{c}{ ``Fair''}
\\ \hline \\
$ \left\{\begin{array}{ll}k\cdot sgn(x) & \text{if } |x|>k\\ x & \text{if } |x|\leq k \end{array}\right. $   & { $\frac{x}{(1+\frac{x^2}{2})^{\frac12}}$} & {$\frac{x}{1+\frac{|x|}c}$}\\
\end{tabular}
\end{center}
}
\end{table}

\section{Lower bounds for relative error algorithms}
 In this section, several lower bounds for relative error approximate PCA are obtained. Our results indicate that it is hard to compute relative error approximate PCA in many cases.
\subsection{Notations}
$e_i^{n}$ denotes a $1\times n$ binary vector whose unique non-zero entry coordinate is on coordinate $i$. $\bar{e_i}^n$ is a $1\times n$ binary vector with unique zero entry coordinated on $i$. $1^n$ denotes a $1\times n$ vector with $n$ ``$1$''. $I_n$ is an identity matrix of size $n$.
\subsection{Lower bounds}
\begin{Theorem}\label{lower1}
Alice has $A^1\in\mathbb{R}^{n\times d}$, Bob has $A^2\in\mathbb{R}^{n\times d}$. Let $A_{i,j}=f(A^1_{i,j}+A^2_{i,j})$. if $f(x)=\Omega(|x|^p)$ and $p>1$, the communication of computing rank-$k$ projection matrix $P$ with constant probability greater than $1/2$ such that $$||A-AP||_F^2\leq(1+\varepsilon)||A-[A]_k||_F^2$$
is at least $\tilde{\Omega}((1+\varepsilon)^{-\frac{2}{p}}n^{1-\frac{1}{p}}d^{1-\frac{4}{p}})$ bits.
\end{Theorem}
We prove it by reduction from the $L_{\infty}$ problem \cite{bar2002information}.
\begin{Theorem}[Theorem 8.1 of \cite{bar2002information}\label{L_inf}]
There are two players. Alice gets $x$, Bob gets $y$, where $x,y$ are length-$m$ vectors. We have $x_i, y_i \in \{0, 1, \dots B\}$ for all $i$.
 There is a promise on the input $(x, y)$ that
 either $\forall i \in [m], |x_i - y_i| \leq 1$ or $\exists ! i, |x_i-y_i|=B$ and $\forall j \neq i, |x_j-y_j| \leq 1$.
 The goal is to determine which case the input is in. If the players want to get the right answer with constant probability greater than $3/4$, then the communication required is $\Omega(\frac{m}{B^2})$ bits.
\end{Theorem}
\begin{proof}[Proof of Theorem \ref{lower1}]
Without loss of generality, we assume $f(x)\equiv |x|^p$. We prove by contradiction. If Alice or Bob can compute the projection matrix $P$ with communication $\tilde{o}((1+\varepsilon)^{-\frac{2}{p}}n^{1-\frac{1}{p}}d^{1-\frac{4}{p}})$ bits, and the success probability is $2/3$, then they can compute $L_{\infty}$ in $o(\frac{m}{B^2})$ bits for $m=n\times d$ and $B=\lceil(2(1+\varepsilon)^2nd^4)^{\frac{1}{2p}}\rceil$.

Assume Alice and Bob have two $m$-dimensional vectors $x,y$ respectively, where $m=n\times d$. $(x,y)$ satisfies the promise mentioned in Theorem \ref{L_inf}. The goal of the players is to find whether there is a coordinate $i$ with $|x_i-y_i|>=\lceil(2(1+\varepsilon)^2nd^4)^{\frac{1}{2p}}\rceil=B$. They agree on the following protocol:

\begin{enumerate}
\item Initially, round $r:=0,n_0:=n$
\item Alice arranges $x$ into an $n\times d$ matrix $A_{(0)}'^1$ and Bob arranges $-y$ into $A_{(0)}'^2$ in the same way.
\item Alice makes $A_{(r)}^1=\left(\begin{array}{cc} A_{(r)}'^1 & 0\\ 0 & B\times I_{k-1} \end{array}\right)$, Bob makes $A_{(r)}^2=\left(\begin{array}{cc} A_{(r)}'^2 & 0\\ 0 & 0 \end{array}\right)$.
\item Alice computes the projection matrix $P$ which satisfies that when the protocol succeeds, $A_{(r)}P$ is the rank-$k$ approximation matrix to $A_{(r)}$, where $(A_{(r)})_{i,j}=|(A_{(r)}^1)_{i,j}+(A_{(r)}^2)_{i,j}|^p$.
\item Alice sorts $e_1^{d+k-1}, \ldots ,e_{d+k-1}^{d+k-1}$ into $e_{i_1}^{d+k-1},...,e_{i_{d+k-1}}^{d+k-1}$ such that $|e_{i_1}^{d+k-1}P|_2\geq \ldots \geq|e_{i_{d+k-1}}^{d+k-1}P|_2$. She finds $i_l\in[d+k-1]$ which satisfies $i_l\leq d\wedge l\leq k$.
\item Alice repeats steps 4-5 $100\lceil\ln(\log_dn)+1\rceil$ times and sets $c$ to be the most frequent $i_l$.
\item Alice rearranges the entries of column $c$ of $A'^1_{(r)}$ into $d$ columns(If $n_r<d$, she fills $d-n_r$ columns with zeros). Thus, Alice gets $A'^1_{(r+1)}$ with $\lceil n_r/d\rceil$ rows. She sends $c$ to Bob.
\item Bob gets $c$. He also rearranges the entries of column $c$ into $d$ columns in the same way. So, he gets $A'^2_{(r+1)}$ with $\lceil n_r/d\rceil$ rows. If $A'^2_{(r+1)}$ has a unique non-zero entry $(A'^2_{(r+1)})_{1,j}$, he sends $j$ to Alice and Alice checks whether $|(A'^1_{(r+1)})_{1,j}+(A'^2_{(r+1)})_{1,j}|=B$. Otherwise, let $r:=r+1,n_{r+1}:=\lceil n_r/d\rceil$, and repeat steps 3-8.
\end{enumerate}

\textbf{Claim:} \textit{If $\exists i,j$ s.t. $|(A_{(r)}^1)_{i,j}+(A_{(r)}^2)_{i,j}|\geq B,j\leq d$, and Alice successfully computes $P$ in step 4, then $i_l$ will be equal to $j$ in step 5.}

Assume Alice successfully computes $P$. Notice that $\sum_{t=1}^{d+k-1} |e_t^{d+k-1}P|_2^2=||P||_F^2=k$. $\forall j\leq d$. If $i_l\not=j$, we have that $|e_j^{d+k-1}P|_2^2\leq \frac{k}{k+1}$. So, if $|(A_{(r)}^1)_{i,j}+(A_{(r)}^2)_{i,j}|\geq B$, we have:
\begin{align*}
& ||A_{(r)}-A_{(r)}P||_F^2\\
& \geq ((A_{(r)})_{i,j}-(A_{(r)}P)_{i,j})^2 \geq(B^p-(\frac{k}{k+1}B^p+d))^2\\
& >(2(1+\varepsilon)\sqrt{n_r}d-d)^2 \geq (1+\varepsilon)^2n_rd^2\\
& >(1+\varepsilon)^2n_rd \geq(1+\varepsilon)^2\min_{X:rank(X)\leq k}||A_{(r)}-X||_F^2 \\
& >||A_{(r)}-A_{(r)}P||_F^2
\end{align*}

Therefore, we have a contradiction. Because the probability that Alice successfully computes $P$ in step 4 is at least $2/3$, according to Chernoff bounds, the probability that more than half of the repetitions in step 6 successfully compute a rank-$k$ approximation is at least $1-\frac1{8\log_dn}$. Due to the claim, if there exists $j$, then $|(A_{(r)}^1)_{i,j}+(A_{(r)}^2)_{i,j}|\geq B$, and $c$ will be equal to $j$ in step 6 with probability at least $1-\frac1{8\log_dn}$. Then, applying a union bound, if there is a coordinate $i$ that $|x_i-y_i|\geq B$, then the probability that the entry will survive until the final check is at least $7/8$. Therefore, the players can output the right answer for the $L_{\infty}$ problem with probability at least $7/8$.

Analysis of the communication cost of the above protocol: there are at most $\lceil\log_d n\rceil$ rounds. In each round, there are $100\lceil\ln(\log_dn)+1\rceil$ repetitions in step 6. Combining with the costs of sending column index $c$ in each round and the final check, the total cost is $\tilde{o}((1+\varepsilon)^{-\frac{2}{p}}n^{1-\frac{1}{p}}d^{1-\frac{4}{p}})$ bits, but according to Theorem \ref{L_inf}, it must be $\tilde{\Omega}((1+\varepsilon)^{-\frac{2}{p}}n^{1-\frac{1}{p}}d^{1-\frac{4}{p}})$. Therefore, it leads to a contradiction.

In the above reduction, a main observation is that $B^p$ is large enough that Alice and Bob can distinguish a large column. Now, consider the function $f(x)=\Omega(|x|^p)$. Let $B=O((2(1+\varepsilon)^2nd^4)^{\frac{1}{2p}})$. Then, the above reduction still works.
\end{proof}

Theorem \ref{lower1} implies that it is hard to design an efficient relative error approximate PCA protocol for $f(\cdot)$ which grows very fast.

\begin{Theorem}\label{lower3}
Alice and Bob have matrices $A^1,A^2\in \mathbb{R}^{n\times d}$ respectively. Let $A_{i,j}=\max(A^1_{i,j},A^2_{i,j})($or $\psi(A^1_{i,j}+A^2_{i,j}))$, where $\psi(x)$ is $\psi$-function of Huber. For $k>1$, the communication of computing rank-$k$ projection matrix $P$ with probability $2/3$ such that $$||A-AP||_F^2\leq(1+\varepsilon)||A-[A]_k||_F^2$$
needs $\tilde{\Omega}(nd)$ bits.
\end{Theorem}
This reduction is from $2$-DISJ problem \cite{razborov1992distributional}. This result gives the motivation that we focus on additive error algorithms for $f(\cdot)$ is $GM$ or $\psi$-function of M-estimators.

\begin{Theorem}{\cite{razborov1992distributional}}\label{2DISJ}
Alice and Bob are given two $n$-bit binary vectors respectively. Either there exists a unique entry in both vectors which is $1$, or for all entries in at least one of the vectors it is $0$. If they want to distinguish these two cases with probability $2/3$, the communication cost is $\Omega(n)$ bits.
\end{Theorem}

\begin{proof}[Proof of Theorem \ref{lower3}]
Specifically, for the Huber $\psi$-function, we assume $\psi(0)=0,\psi(1)=1,\psi(2)=1$.
Alice and Bob are given $(n\times d)$-bit binary vectors $x$ and $y$ respectively. $x$ and $y$ satisfy the promise in Theorem \ref{2DISJ}.  If they can successfully compute a projection matrix $P$ mentioned in Theorem \ref{lower3} with probability $2/3$, then they can agree on the following protocol to solve 2-DISJ:
\begin{enumerate}
\item Initialize round $r:=0,n_0:=n$
\item The players flip all the bits in $x$, $y$ and arrange flipped $x$, $y$ in $n\times d$ matrices $A'^1_{(0)}$, $A'^2_{(0)}$ respectively.
\item Alice makes $A_{(r)}^1=\left(\begin{array}{cc}A_{(r)}'^1 & 0\\ 1^d & 0\\ 0 & I_{k-2} \end{array}\right)$, Bob makes $A_{(r)}^2=\left(\begin{array}{cc}A_{(r)}'^2 & 0\\ 0 & 0\\ 0 & 0 \end{array}\right)$.
\item Let $(A_{(r)})_{i,j}=\max((A_{(r)}^1)_{i,j},(A_{(r)}^2)_{i,j})$ (or $(A_{(r)})_{i,j}=\psi((A_{(r)}^1)_{i,j}+(A_{(r)}^2)_{i,j})$). Alice computes $P$ as mentioned in the statement of Theorem 5. If the protocol succeeds, $A_{(r)}P$ is the rank-$k$ approximation matrix to $A_{(r)}$.
\item Alice finds $l\in[d]$ such that $(\bar{e_l}^d\ 0)P=(\bar{e_l}^d\ 0)$
\item Alice repeats steps 4-5 $100\lceil\ln(\log_dn)+1\rceil$ times and sets $c$ to be the most frequent $l$.
\item Alice rearranges the entries of column $c$ of $A'^1_{(r)}$ into $A'^1_{(r+1)}$ whose size is $\lceil n_r/d\rceil\times d$. She sends $c$ to Bob.
\item Bob also rearranges the entries of column $c$ into $d$ columns in the same way. Thus, he gets $A'^2_{(r+1)}$. If $A'^2_{(r+1)}$ has a unique zero entry $(A'^2_{(r+1)})_{1,j}$, he sends $j$ to Alice and Alice checks whether $(A'^1_{(r+1)})_{1,j}=0$. Otherwise, Let $r:=r+1,n_{(r+1)}:=\lceil n_r/d\rceil$, repeat steps 3-8.
\end{enumerate}

An observation is that $\forall r$, there is at most one pair of $i,j$ such that $\max((A_{(r)}^1)_{i,j},(A_{(r)}^2)_{i,j})(\text{or } \psi((A_{(r)}^1)_{i,j}+(A_{(r)}^2)_{i,j}))=0$. Therefore, the rank of $A_{(r)}$ is at most $k$. If Alice successfully computes $P$,
$$||A_{(r)}-A_{(r)}P||_F^2\leq (1+\varepsilon)||A_{(r)}-[A_{(r)}]_k||_F^2=0$$
The row space of $P$ is equal to the row space of $A_{(r)}$. Notice that if $(\bar{e_i}^d\ 0)$ is in the row space of $A_{(r)}$, for $j\not=i$, $(\bar{e_j}^d\ 0)$ cannot be in the row space of $A_{(r)}$. If Alice succeeds in step 4, she will find at most one $l$ in step 5. Furthermore, if $\exists i,j$ s.t. $(A_{(r)})_{i,j}=0,j\leq d$, $(\bar{e_j}^d\ 0)$ must be in the row space of $A_{(r)}$. Then $l$ will be equal to $j$ in step 5. Similar to the proof of Theorem \ref{lower1}, according to a Chernoff bound and a union bound, if there is a joint coordinate, Alice and Bob will find the coordinate with probability at least $7/8$.

The total communication includes the communication of at most $100\lceil\ln(\log_dn)+1\rceil\times\log_d n$ times of the computation of a rank-$k$ approximation and the communication of sending several indices in step 7-8. Due to Theorem \ref{2DISJ}, the total communication is $\Omega(nd)$ bits. Thus, computing $P$ needs $\tilde{\Omega}(nd)$ bits.
\end{proof}

Finally, the following reveals the relation between lower bounds of relative error algorithms and $\varepsilon$.
\begin{Theorem}\label{lower2}
Alice and Bob have $A^1,A^2\in\mathbb{R}^{n\times d}$ respectively. Let $A_{i,j}=f(A^1_{i,j}+A^2_{i,j})$, where $f(x)=x^p($or $|x|^p),p\not=0$. Suppose $1/\varepsilon^2\leq 2n$, the communication of computing a rank-$k$ $P$ with probability $2/3$ such that
$$||A-AP||_F^2\leq(1+\varepsilon)||A-[A]_k||_F^2$$
needs $\Omega(1/\varepsilon^2)$ bits.
\end{Theorem}
To prove this, we show that if relative error approximate PCA can be done in $o(1/\varepsilon^2)$ bits communication, then the GHD problem in \cite{chakrabarti2012optimal} can be solved efficiently.
\begin{Theorem}{\cite{chakrabarti2012optimal}}\label{gaphamming}
Each of Alice and Bob has an $n$-bit vector. There is a promise: either the inputs are at Hamming distance less than $n/2 - c\sqrt{n}$ or greater than $n/2 + c\sqrt{n}$. If they want to distinguish the above two cases with probability $2/3$, the communication required is $\Omega(n)$ bits.
\end{Theorem}
\begin{proof}[Proof of Theorem \ref{lower2}]
Without loss of generality, suppose $f(x)\equiv x$. Assume Alice and Bob have $x,y\in\{-1,1\}^{1/\varepsilon^2}$ respectively. $ \langle x,y \rangle$ denotes the inner product between $x$ and $y$. There is a promise that either $\langle x,y \rangle <-2/\varepsilon$ or $\langle x,y \rangle >2/\varepsilon$ holds. Alice and Bob agree on the following protocol to distinguish the above two cases:

\begin{enumerate}
\item Alice constructs $(1/\varepsilon^2+k)\times (k+1)$ matrix $A^1$ and Bob constructs $A^2$:
\begin{align*}
&A^1=\left(\begin{array}{ccccc}x_1\varepsilon&0&0&...&0\\x_2\varepsilon&0&0&...&0\\...&...&...&...&...\\x_{\frac1{\varepsilon^2}}\varepsilon&0&0&...&0\\
0&\sqrt{2}&0&...&0\\0&0&\frac{\sqrt{2(1+\varepsilon)}}{\varepsilon}&...&0\\...&...&...&...&...\\0&0&0&...&\frac{\sqrt{2(1+\varepsilon)}}{\varepsilon}\end{array}\right)\\
&A^2=\left(\begin{array}{ccccc}y_1\varepsilon&0&0&...&0\\y_2\varepsilon&0&0&...&0\\...&...&...&...&...\\y_{\frac1{\varepsilon^2}}\varepsilon&0&0&...&0\\
0&0&0&...&0\\0&0&0&...&0\\...&...&...&...&...\\0&0&0&...&0\end{array}\right)
\end{align*}
\item Let $A=A^1+A^2$. Alice computes the rank-$k$ projection matrix $P$ such that $$||A-AP||_F^2\leq(1+\varepsilon)||A-[A]_k||_F^2$$
\item Let $u$ be the first row of $(I_{k+1}-P)$. Let $v$ be $u/|u|_2$.
\item Alice checks whether $v_1^2<\frac12(1+\varepsilon)$, if yes, return the case $\langle x,y \rangle >2/\varepsilon$, otherwise, return the case $\langle x,y \rangle <-2/\varepsilon$
\end{enumerate}

Assume that Alice successfully computes $P$. Because the rank of $A$ is at most $k+1$,
$$||A-AP||_F^2=||A(I_{k+1}-P)||_F^2=|Av|_2^2=v^TA^TAv$$
Notice that
$$A^TA=\left(\begin{array}{ccccc}|x+y|_2^2\varepsilon^2&0&0&...&0\\0&2&0&...&0\\0&0&\frac{2(1+\varepsilon)}{\varepsilon^2}&...&0\\...&...&...&...&...\\0&0&0&...&\frac{2(1+\varepsilon)}{\varepsilon^2}\end{array}\right)$$
We have $||A-[A]_k||_F^2\leq 2$. Then, $$\frac{2(1+\varepsilon)}{\varepsilon^2}\sum_{i=3}^{k+1}v_i^2=\frac{2(1+\varepsilon)}{\varepsilon^2}(1-v_1^2-v_2^2)<2(1+\varepsilon)$$
We have $v_1^2+v_2^2>1-\varepsilon^2$.
When $\langle x,y \rangle > 2/\varepsilon$, $|x+y|_2^2\varepsilon^2\geq2+4\varepsilon$
\begin{align*}
&(1+\varepsilon)||A-[A]_k||_F^2=2(1+\varepsilon)\\
&>v_1^2(2+4\varepsilon)+2v_2^2>2-2\varepsilon^2+4\varepsilon v_1^2
\end{align*}
So, $v_1^2<\frac12(1+\varepsilon)$.

When $\langle x,y \rangle <-2/\varepsilon$, $|x+y|_2^2\varepsilon^2\leq2-4\varepsilon$
\begin{align*}
& (1+\varepsilon)||A-[A]_k||_F^2=|x+y|_2^2\varepsilon^2(1+\varepsilon)\\
&>v_1^2|x+y|_2^2\varepsilon^2+2v_2^2=2(v_1^2+v_2^2)-(2-|x+y|_2^2\varepsilon^2)v_1^2\\
&>2(1-\varepsilon^2)-(2-|x+y|_2^2\varepsilon^2)v_1^2
\end{align*}
So, $v_1^2>\frac{2(1-\varepsilon^2)-|x+y|_2^2\varepsilon^2(1+\varepsilon)}{2-|x+y|_2^2\varepsilon^2}\geq\frac12(1+\varepsilon)$.
Therefore, Alice can distinguish these two cases. The only communication cost is for computing the projection matrix $P$. This cost is $o(1/\varepsilon^2)$ bits, which contradicts to the $\Omega(1/\varepsilon^2)$ bits of the gap Hamming distance problem.

Notice in step 1 of the protocol, if we replace $x_i\varepsilon$, $y_i\varepsilon$, $\sqrt{2}$, $\frac{\sqrt{2(1+\varepsilon)}}{\varepsilon}$ with $\frac12x_i(2\varepsilon)^\frac1p$, $\frac12y_i(2\varepsilon)^\frac1p$, $2^\frac1{2p}$, $(\frac{\sqrt{2(1+\varepsilon)}}{\varepsilon})^\frac{1}{p}$ respectively, the above reduction also holds more generally.
\end{proof}

\section{Experiments and evaluations}
 We ran several experiments to test the performance of our algorithm in different applications, including Gaussian random Fourier features \cite{rahimi2007random}, P-norm pooling \cite{boureau2010theoretical} (square-root pooling \cite{yang2009linear}) and robust PCA. We measured actual error given a bound on the total communication.

\textbf{Setup: } We implement our algorithms in C++. We use multiple processes to simulate multiple servers. For each experiment, we give a bound of the total communication. That means we will adjust some parameters including
\begin{enumerate}
\item The number $r$ of sampled rows in Algorithm \ref{PCA}.
\item The number $t$ of repetitions in Algorithm \ref{improvedheavyhitters} and number of hash buckets.
\item Parameters $B$, $W$, and number $e$ of repetitions in Algorithm \ref{estimator}.
\end{enumerate}
to guarantee the ratio of the amount of total communication to the sum of local data sizes is limited. We measure both actual additive error $\left|||A-AP||_F^2-||A-[A]_k||_F^2\right|/||A||_F^2$ and actual relative error $\left|||A-AP||_F^2/||A-[A]_k||_F^2\right|$ for various limitations of the ratios, where $P$ is the output rank-$k$ projection matrix of our algorithm.

\textbf{Datasets: } We ran experiments on datasets in \cite{bache2013uci,rahimi2007random,boureau2010theoretical}. We chose \textit{Forest Cover} ($522000\times 5000$ Fourier features) and \textit{KDDCUP99} ($4898431\times 50$ Fourier features) which are mentioned in \cite{rahimi2007random} to examine low rank approximation for Fourier features. \textit{Caltech-101} ($9145\times 256$ features) and \textit{Scenes} ($4485\times256$ features) mentioned in \cite{boureau2010theoretical} are chosen to evaluate approximate PCA on features of images after generalized mean pooling. Finally, we chose \textit{isolet} ($ 1559\times 617$), which is also chosen by \cite{ding2006r}, to evaluate robust PCA with the $\psi$-function of the Huber M-estimator.

\textbf{Methodologies: } For Gaussian random Fourier features, we randomly distributed the original data to different servers. For datasets \textit{Caltech-101} and \textit{Scenes}, we generated matrices similar to \cite{boureau2010theoretical}: densely extract SIFT descriptors of $64\times 64$ pixels patch every $32$ pixels; use k-means to generate a codebook with size $256$; and generate a $1$-of-$256$ code for each patch. We distributed the $1$-of-$256$ binary codes to different servers. Each server locally pooled the binary codes of the same image. Thus, the global matrix can be obtained by pooling across servers. When doing pooling, we focused on average pooling (P=1), square-root pooling (P=2), and P-norm pooling for P=5 and P=20 for simulating max pooling. Finally, we evaluated robust PCA. To simulate the noise, we randomly changed values of $50$ entries of the feature matrix of \textit{isolet} to be extremely large and then we arbitrarily partitioned the matrix into different servers. Since we can arbitrarily partition the matrix, a server may not know whether an entry is abnormally large. We used the Huber $\psi$-function to filter out the large entries. We ran each experiment $5$ times and measured the average error. The number of servers is $10$ for \textit{Forest Cover}, \textit{Scenes} and \textit{isolet}, and $50$ for \textit{KDDCUP99} and \textit{Caltech-101}. We compared our experimental results with our theoretical predictions. If we sample $r$ rows, we predict the additive error will be $k^2/r$. We also compared the accuracy when we gave different bounds on the ratio of amounts of total communication to the sum of local data sizes. The results are shown in Figure \ref{experimentresult1} and Figure \ref{experimentresult2}. As shown, in practice, our algorithm performed better than its theoretical prediction.

\begin{figure}[b!]
\noindent\begin{tabular}{cc}
  \includegraphics[width=0.245\textwidth]{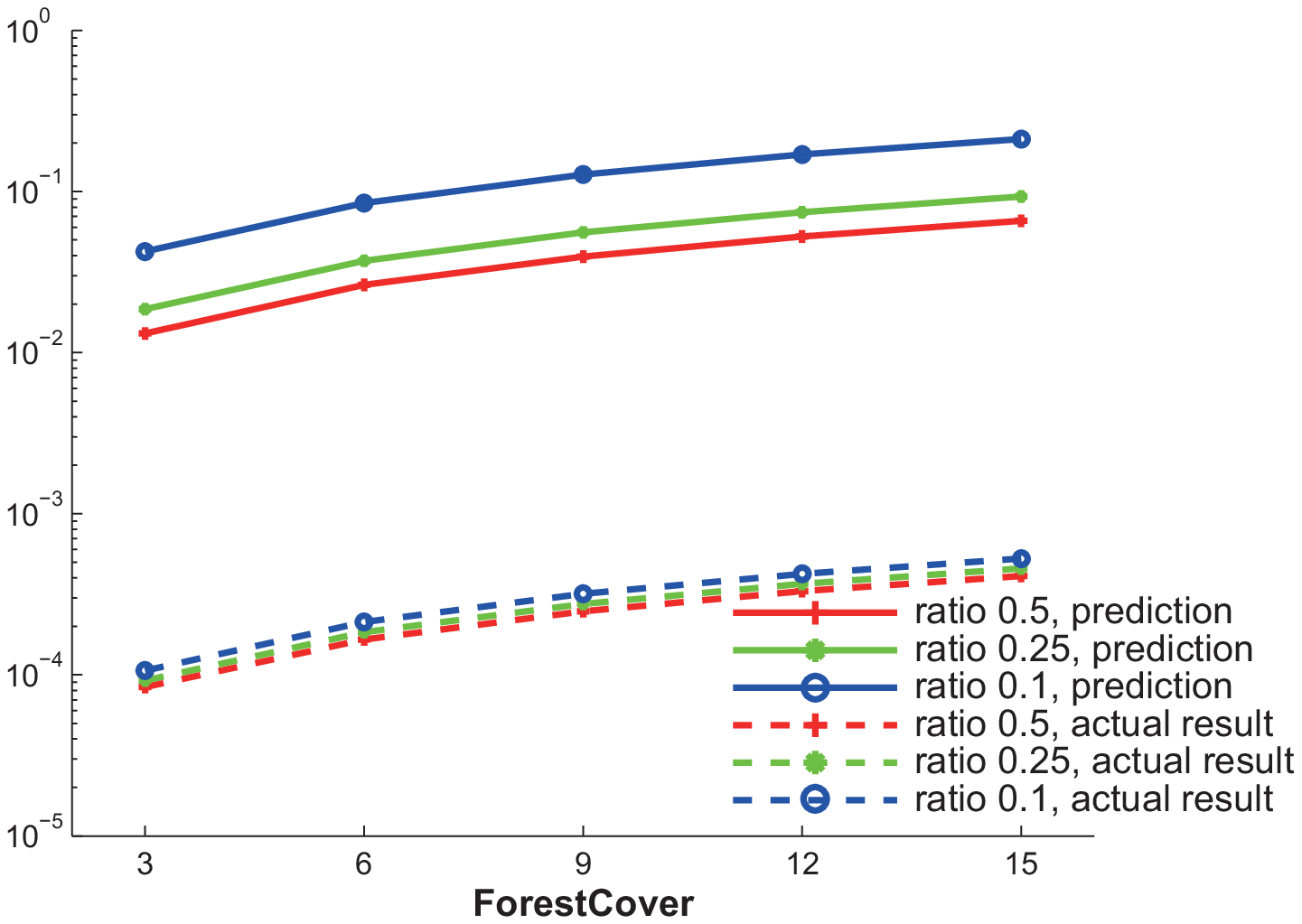}&
  \includegraphics[width=0.245\textwidth]{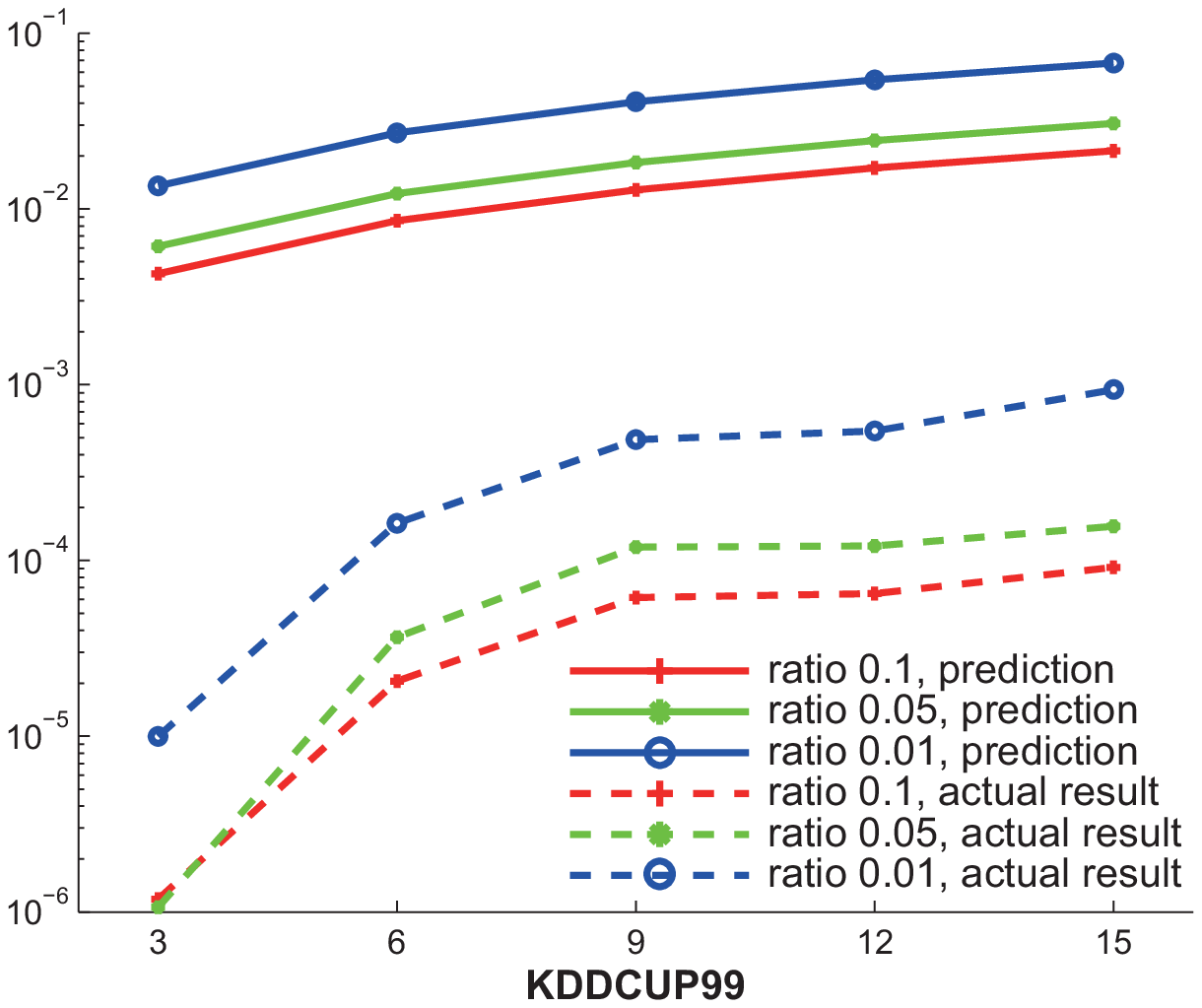}\\
  \includegraphics[width=0.245\textwidth]{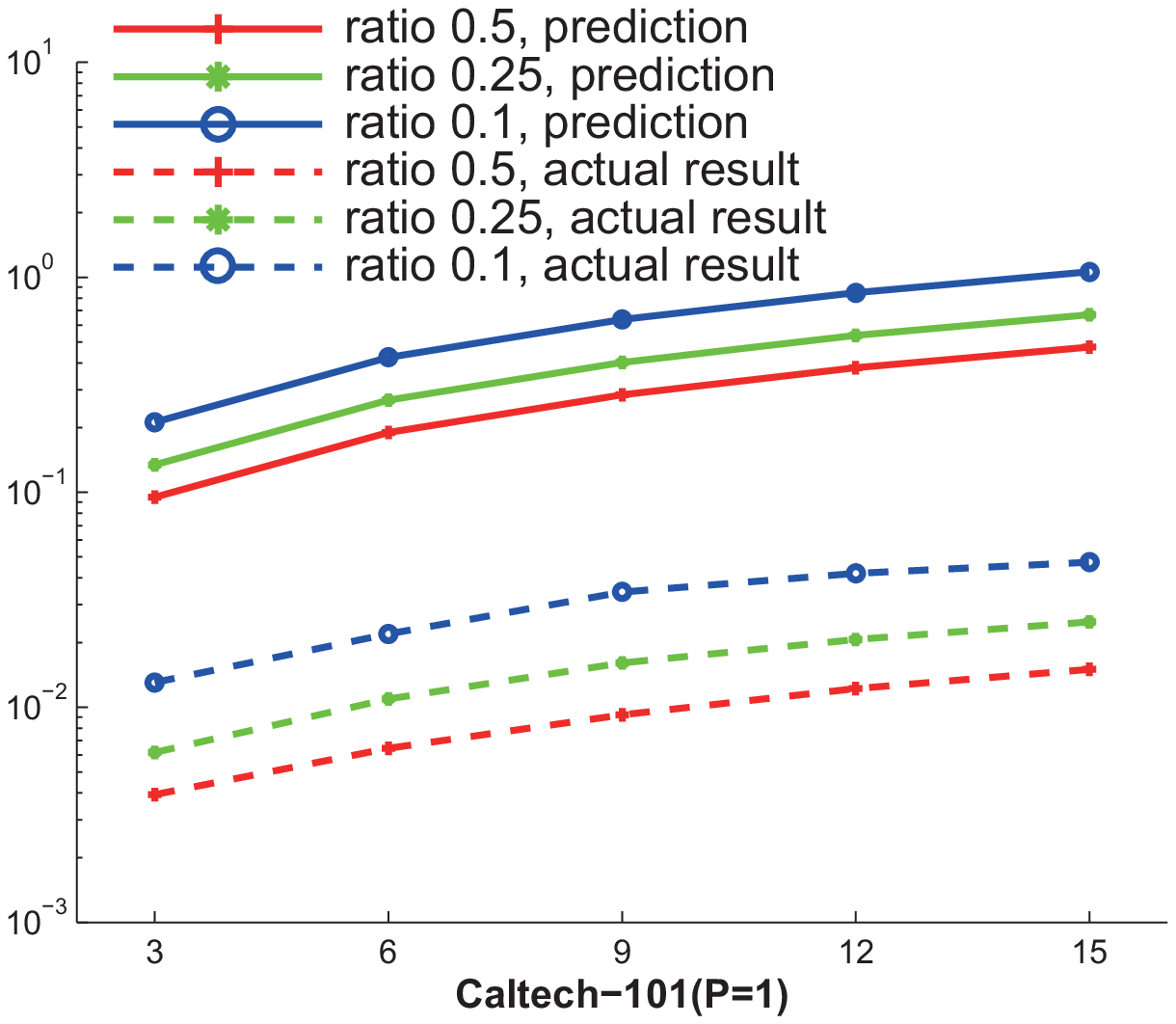}&
  \includegraphics[width=0.245\textwidth]{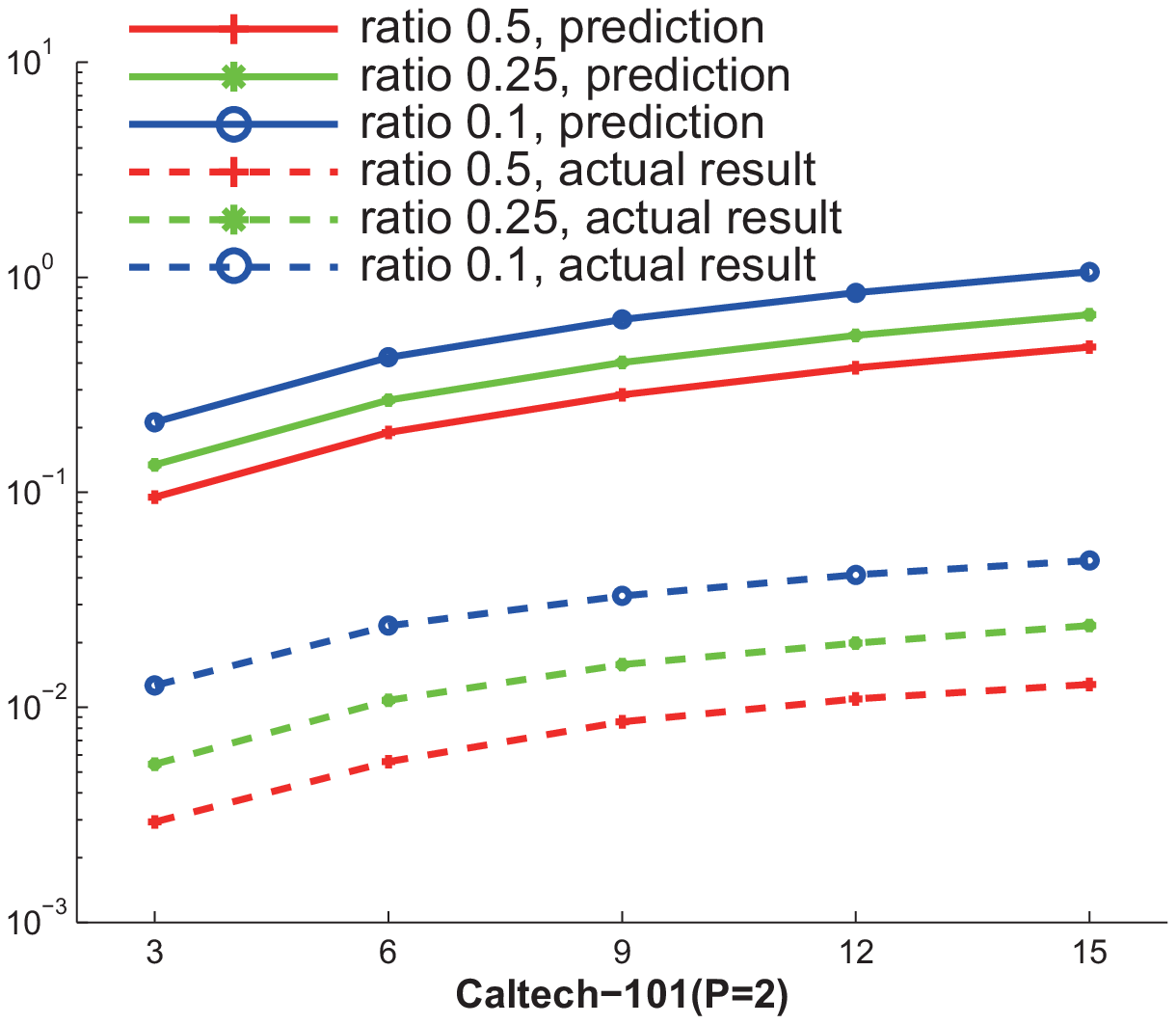}\\
  \includegraphics[width=0.245\textwidth]{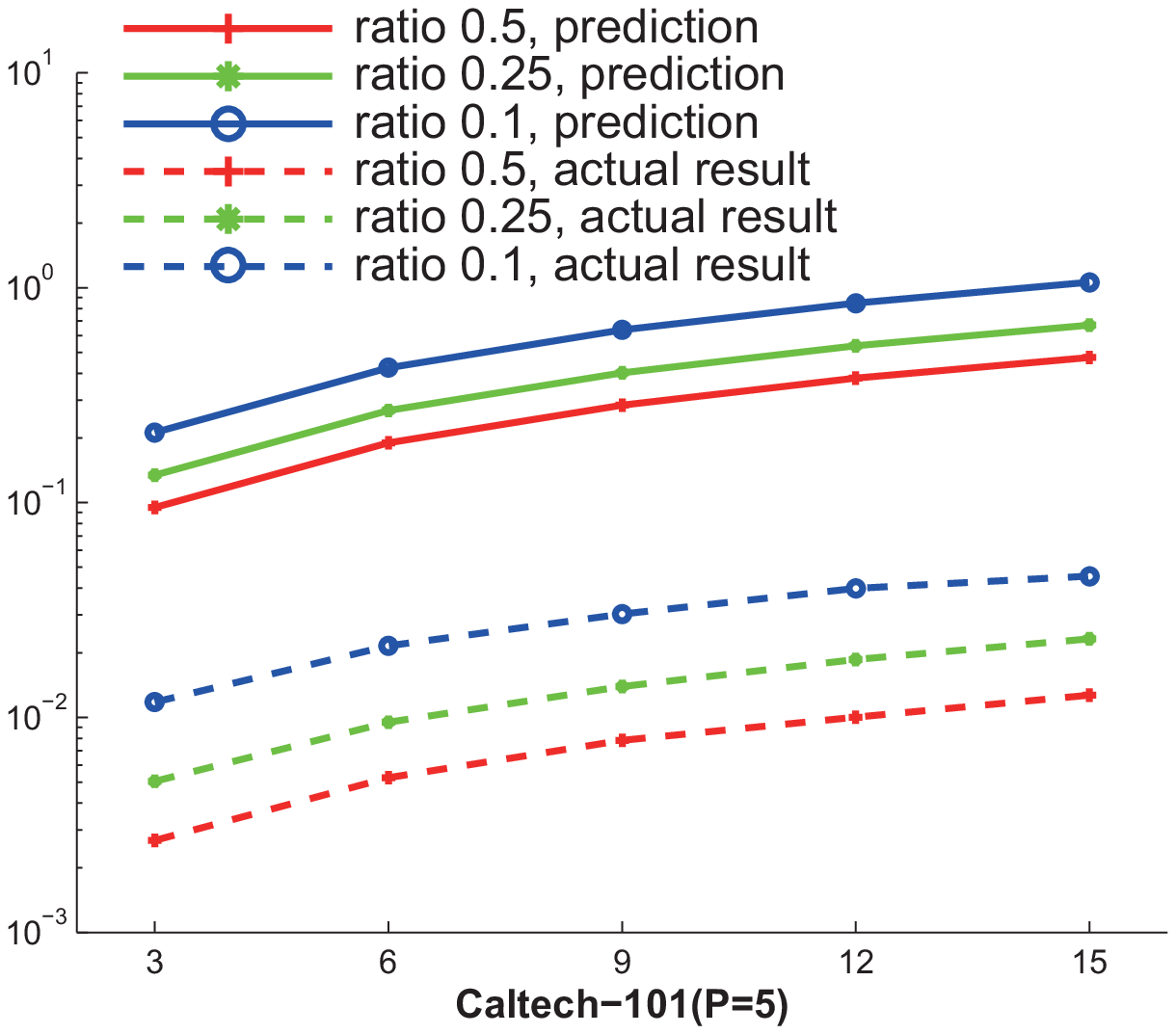}&
  \includegraphics[width=0.245\textwidth]{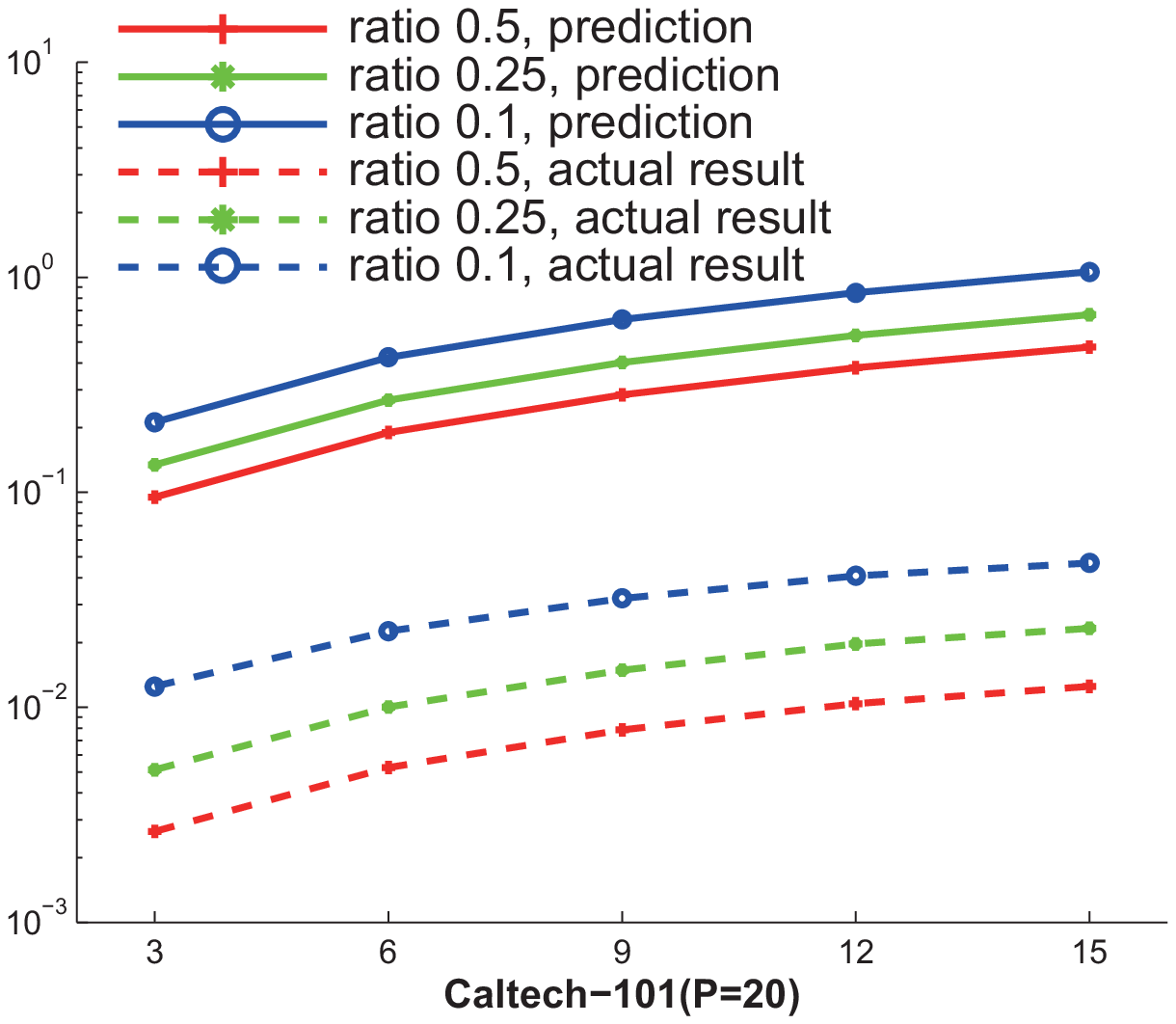}\\
\end{tabular}
\end{figure}
\begin{figure}
\noindent\begin{tabular}{cc}
  \includegraphics[width=0.245\textwidth]{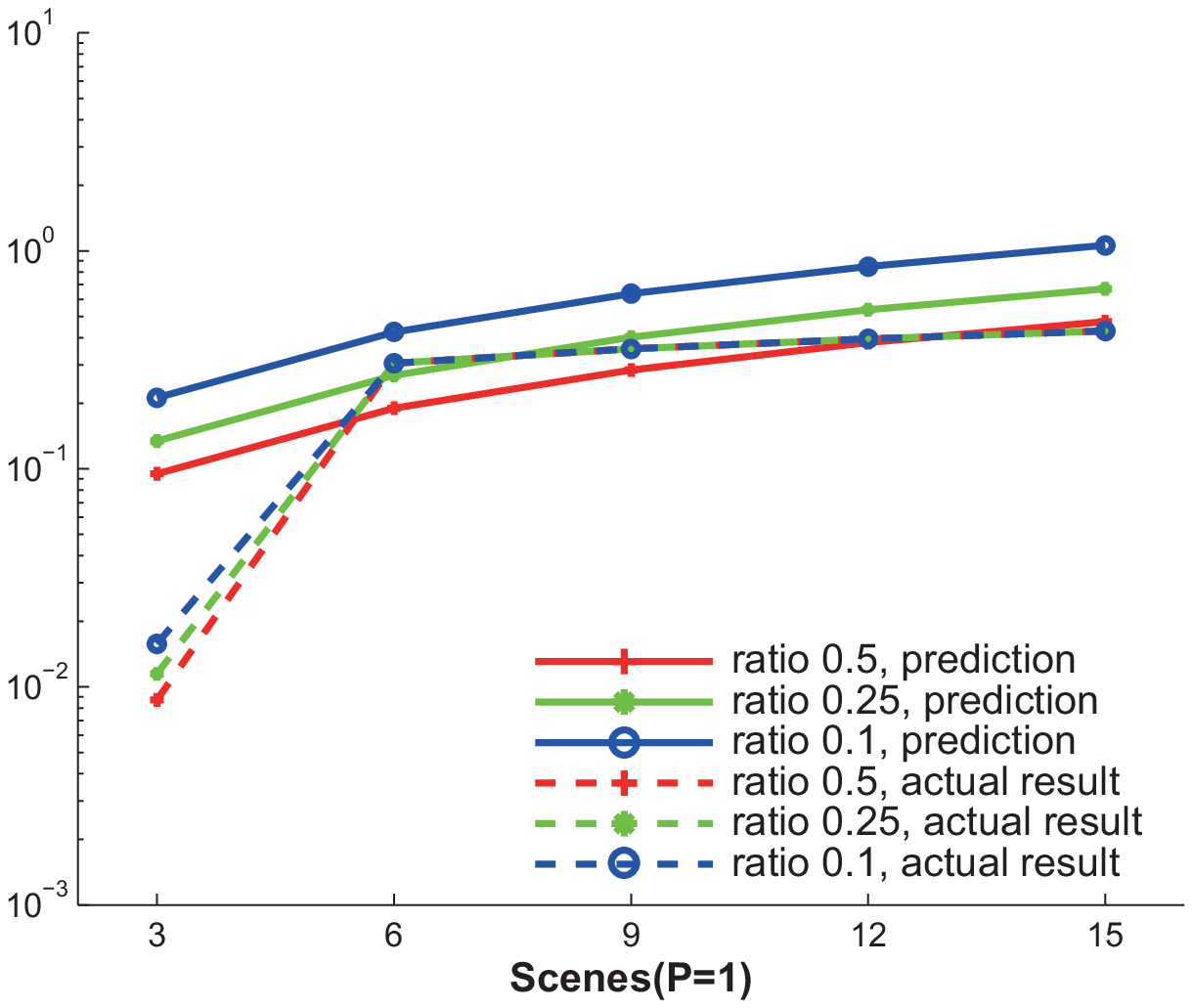}&
  \includegraphics[width=0.245\textwidth]{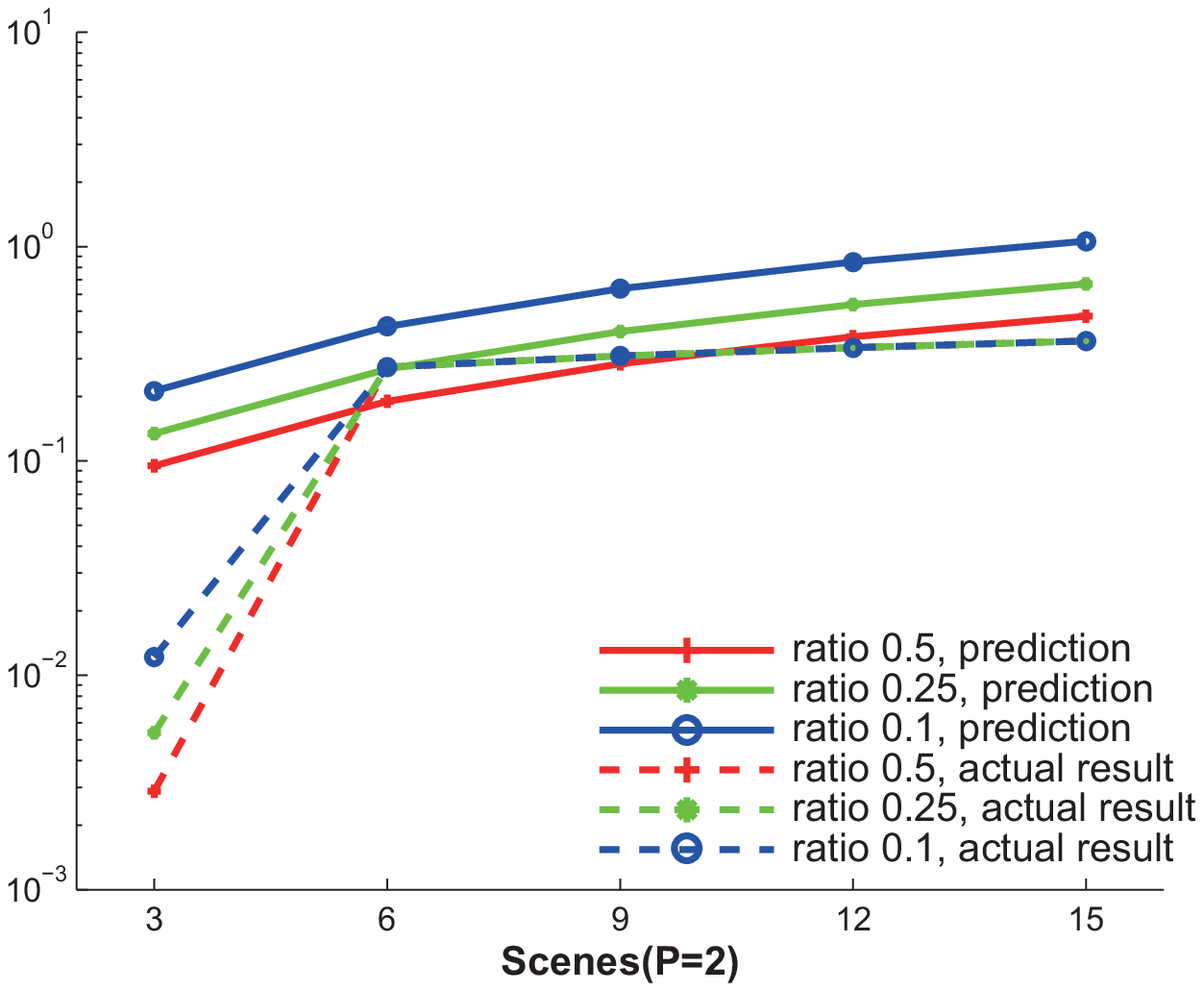}\\
  \includegraphics[width=0.245\textwidth]{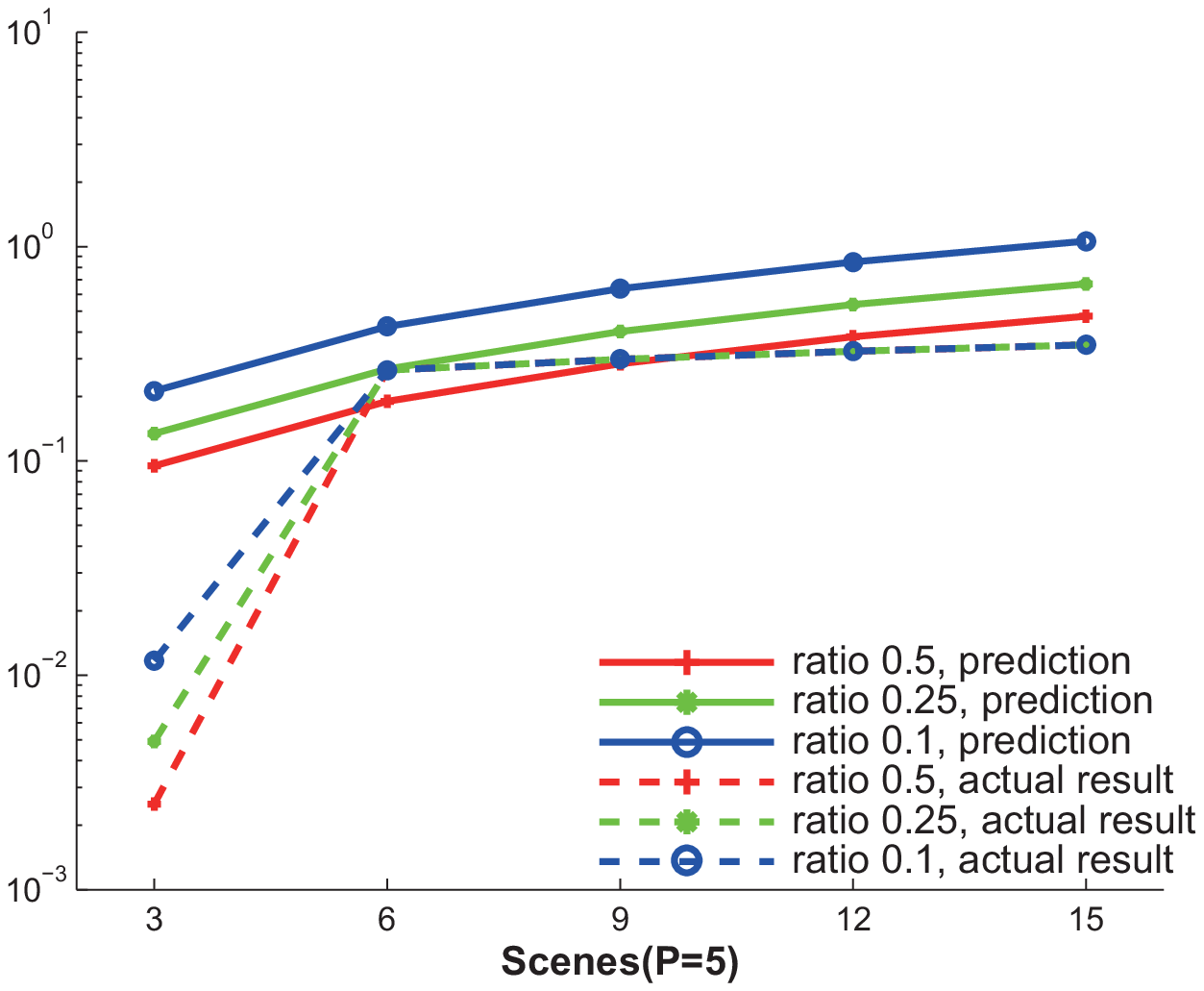}&
  \includegraphics[width=0.245\textwidth]{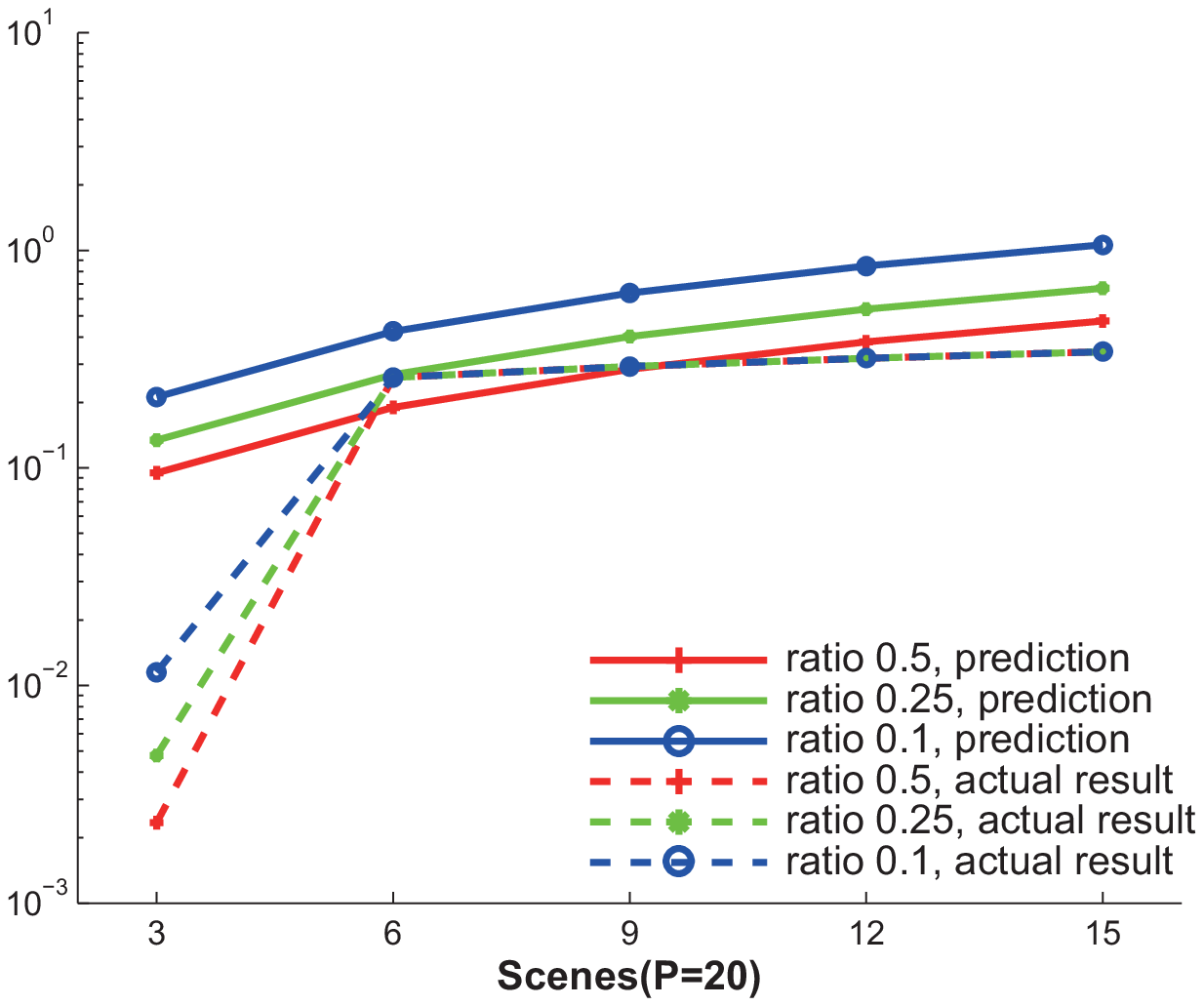}\\
  \includegraphics[width=0.245\textwidth]{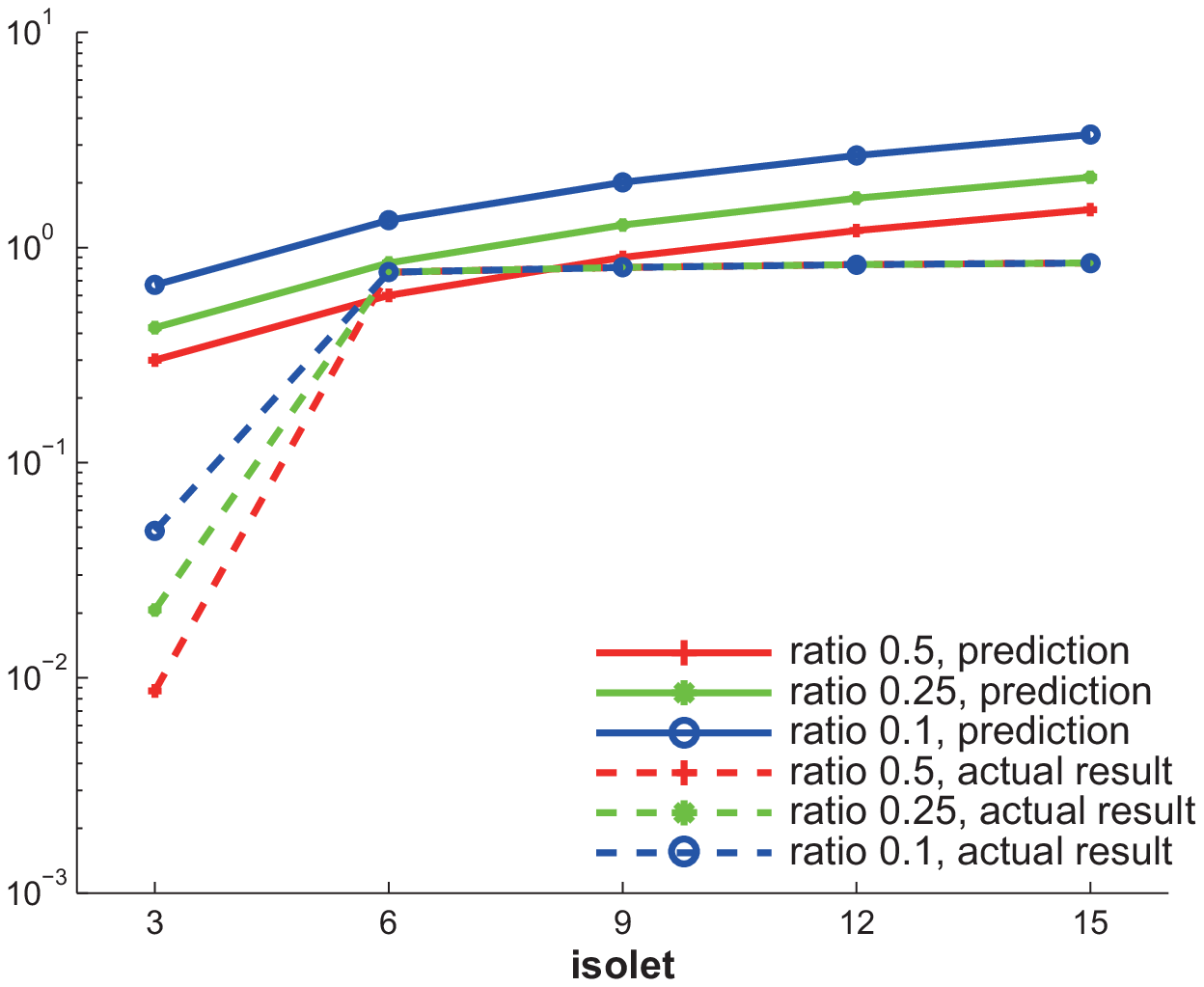}\\
\end{tabular}
\centering
\caption{Additive error v.s. projection dimension. For \textit{Forest Cover} and \textit{KDDCUP99}, we evaluate PCA for Gaussian Fourier features. For \textit{Caltech-101} and \textit{Scenes}, we evaluate on P-norm pooling features for different P. We evaluate robust PCA on \textit{isolet}.} \label{experimentresult1}
\end{figure}

\begin{figure}[b!]
\noindent\begin{tabular}{cc}
  \includegraphics[width=0.245\textwidth]{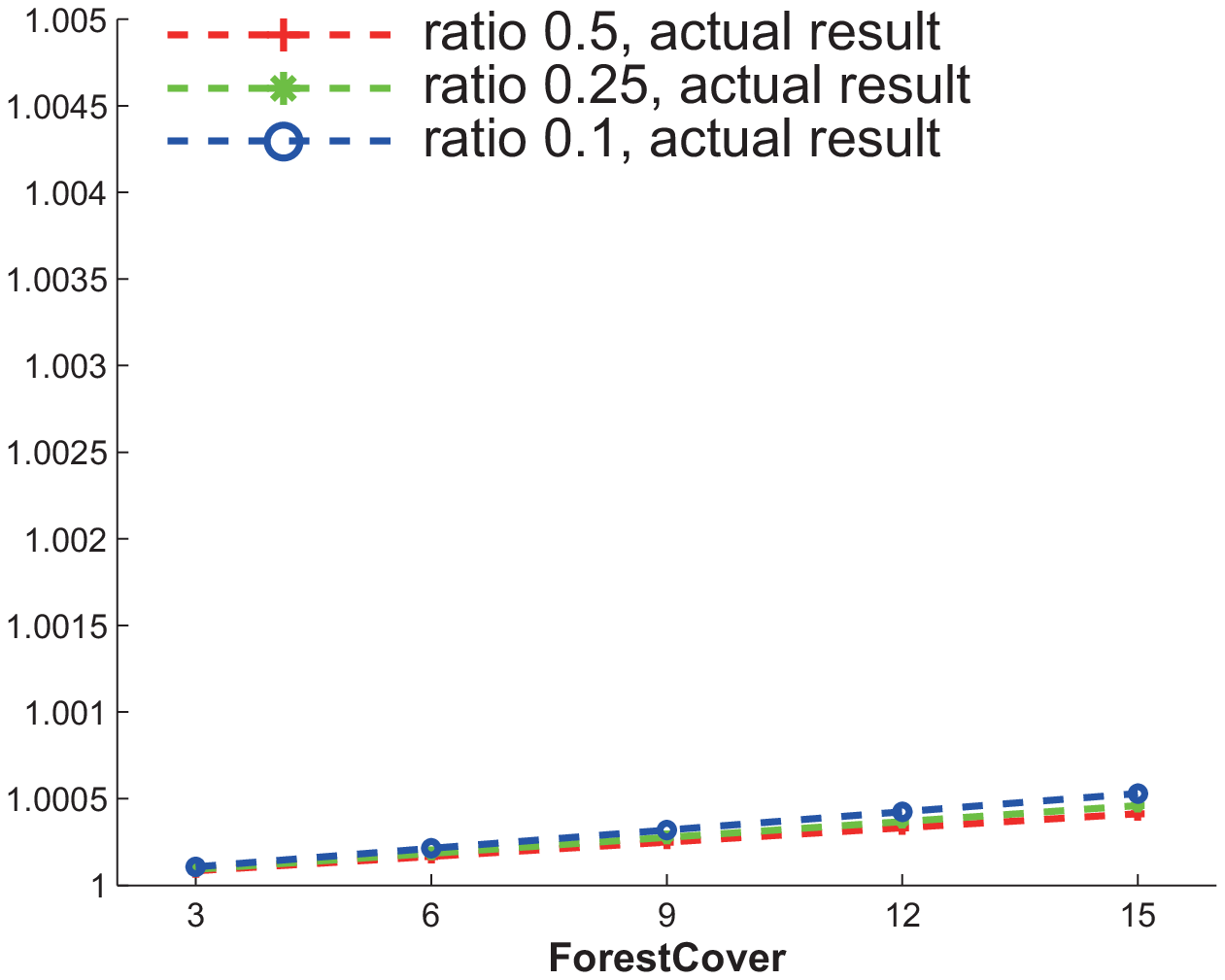}&
  \includegraphics[width=0.245\textwidth]{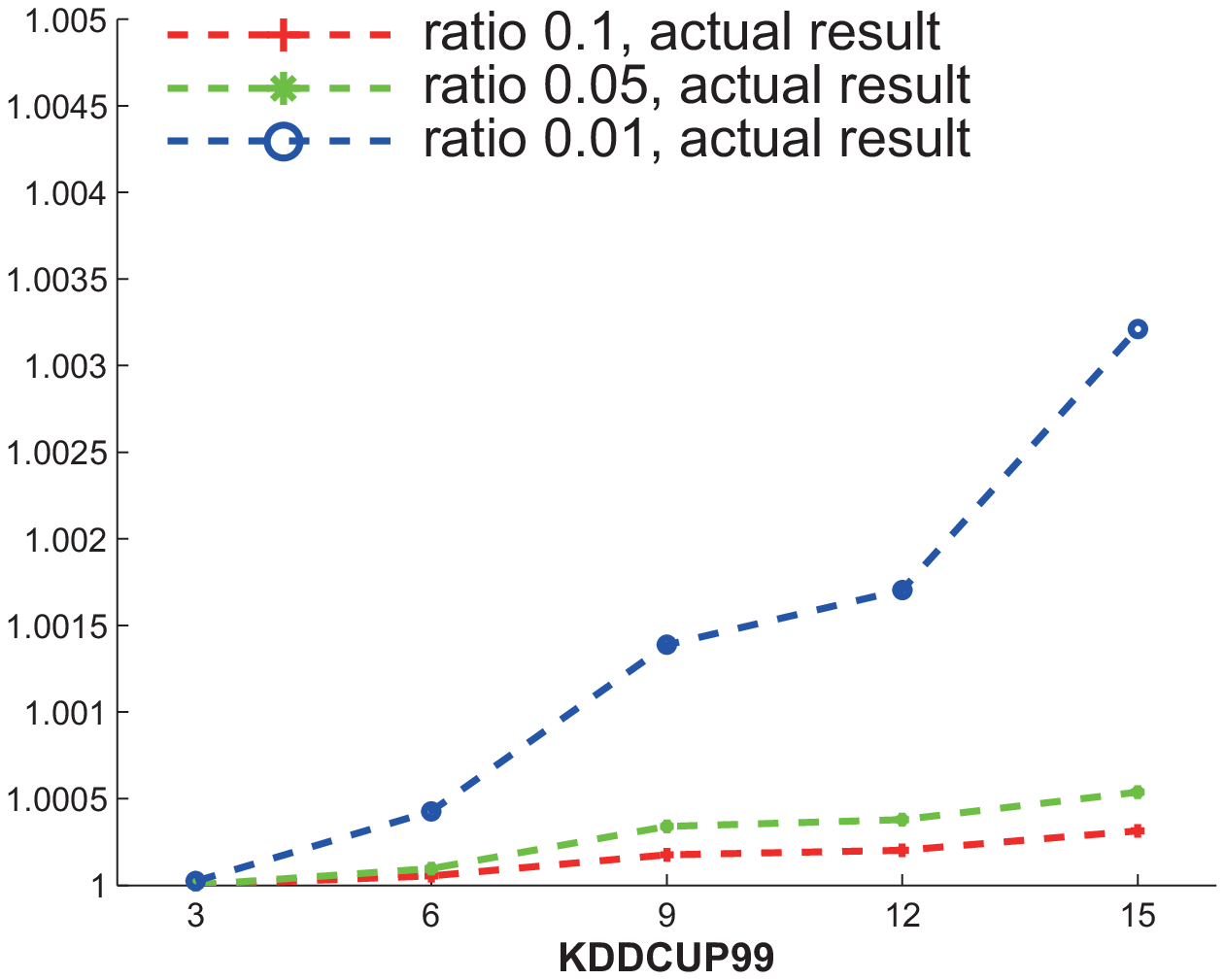}\\
  \includegraphics[width=0.245\textwidth]{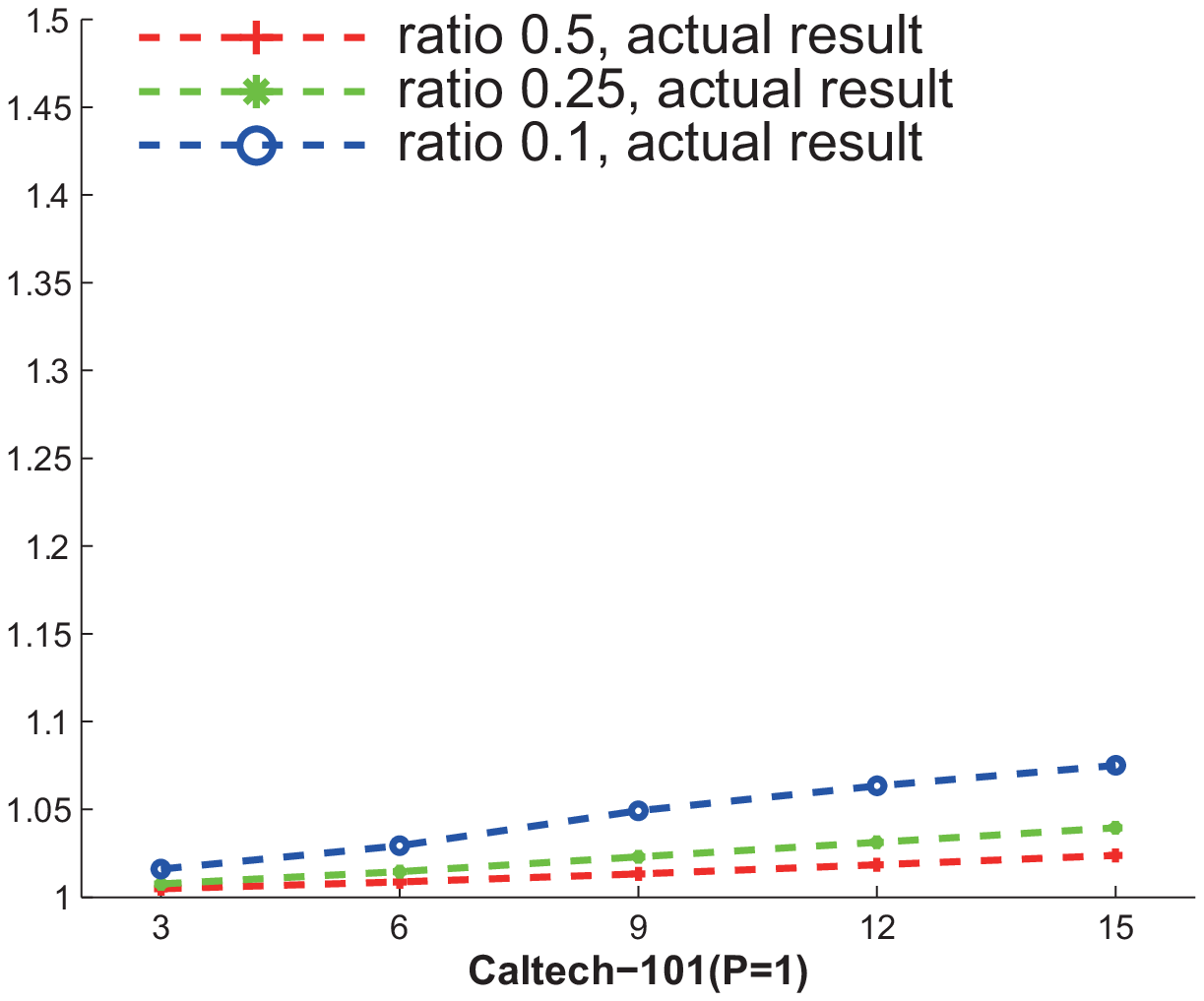}&
  \includegraphics[width=0.245\textwidth]{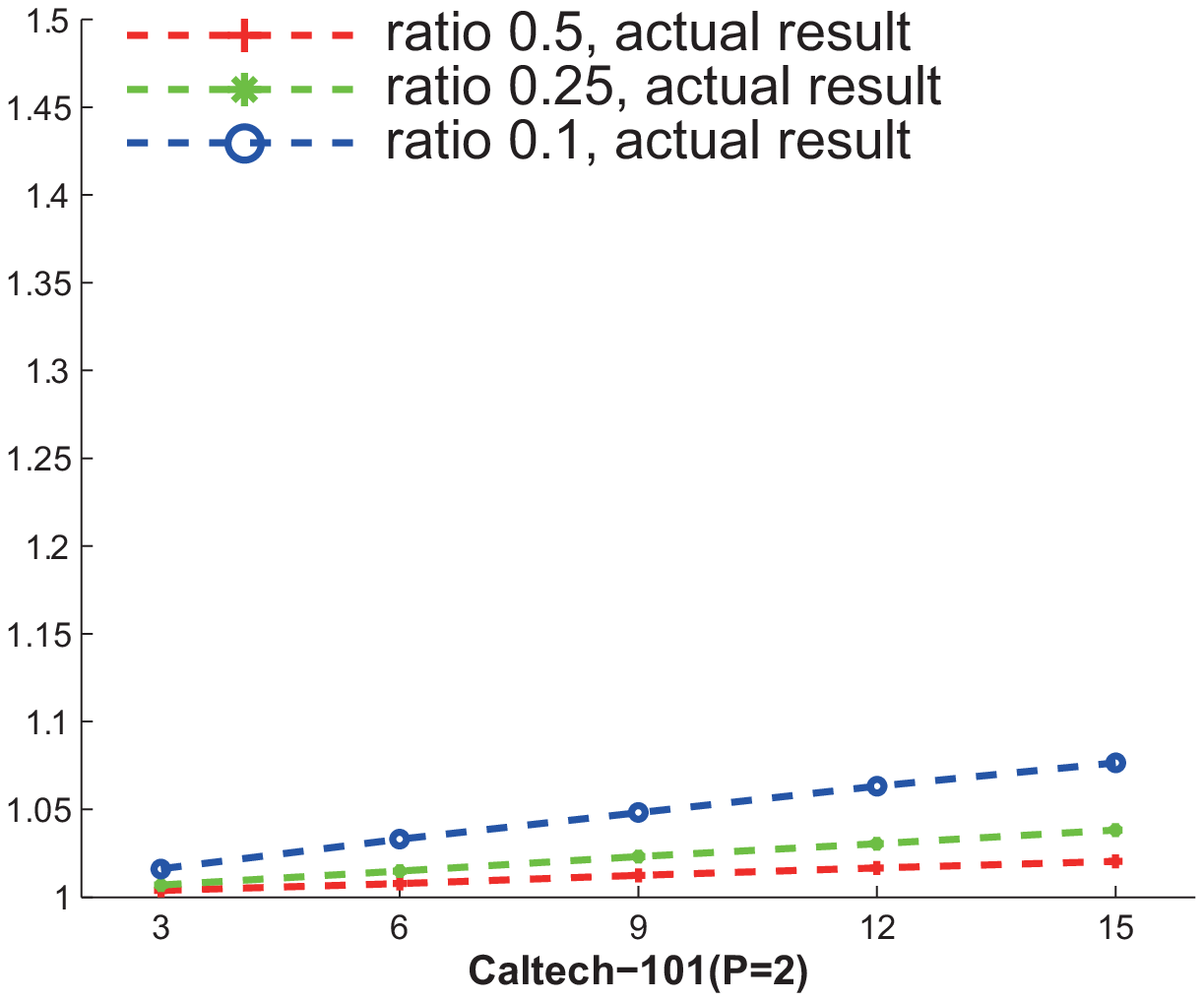}\\
  \includegraphics[width=0.245\textwidth]{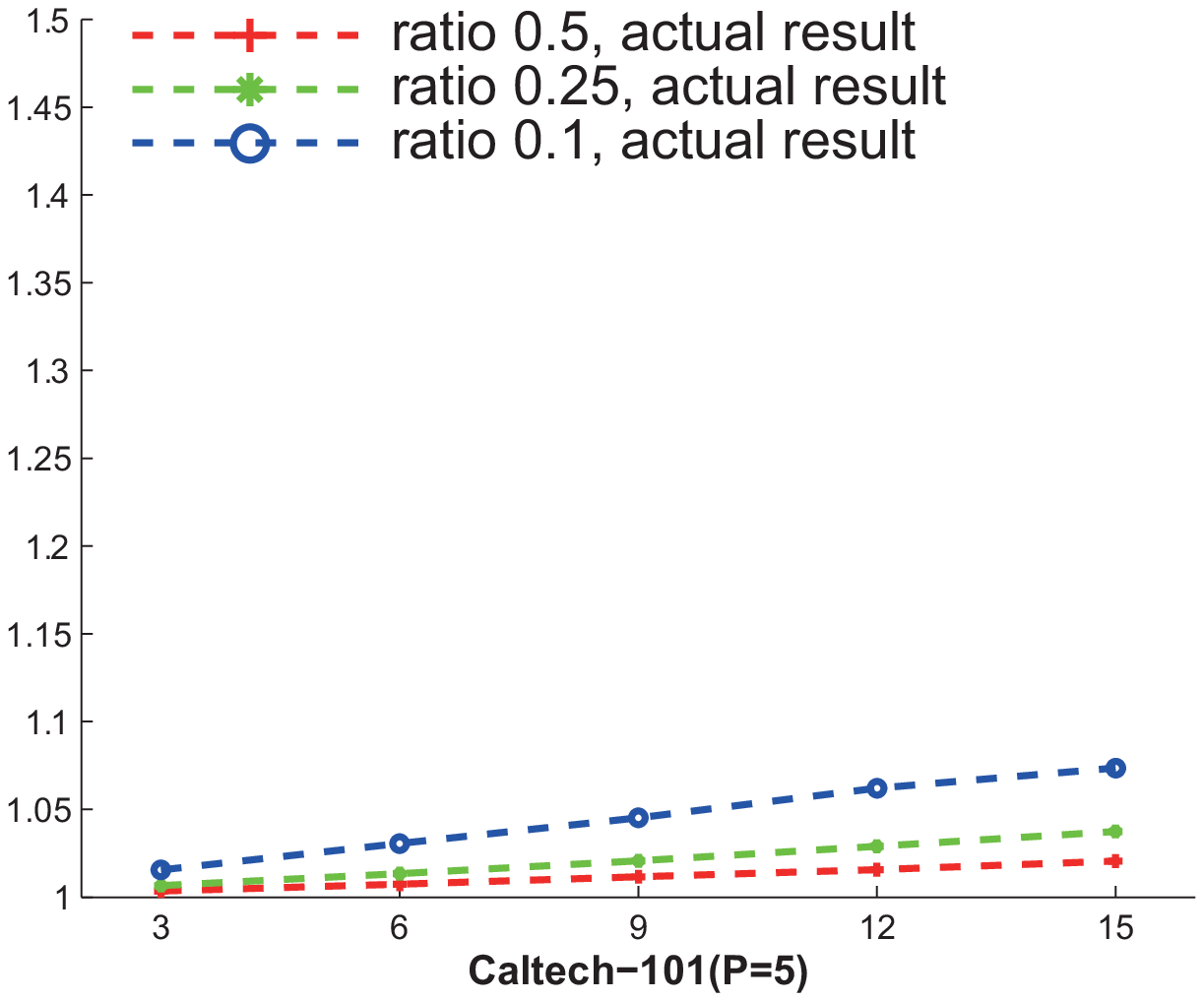}&
  \includegraphics[width=0.245\textwidth]{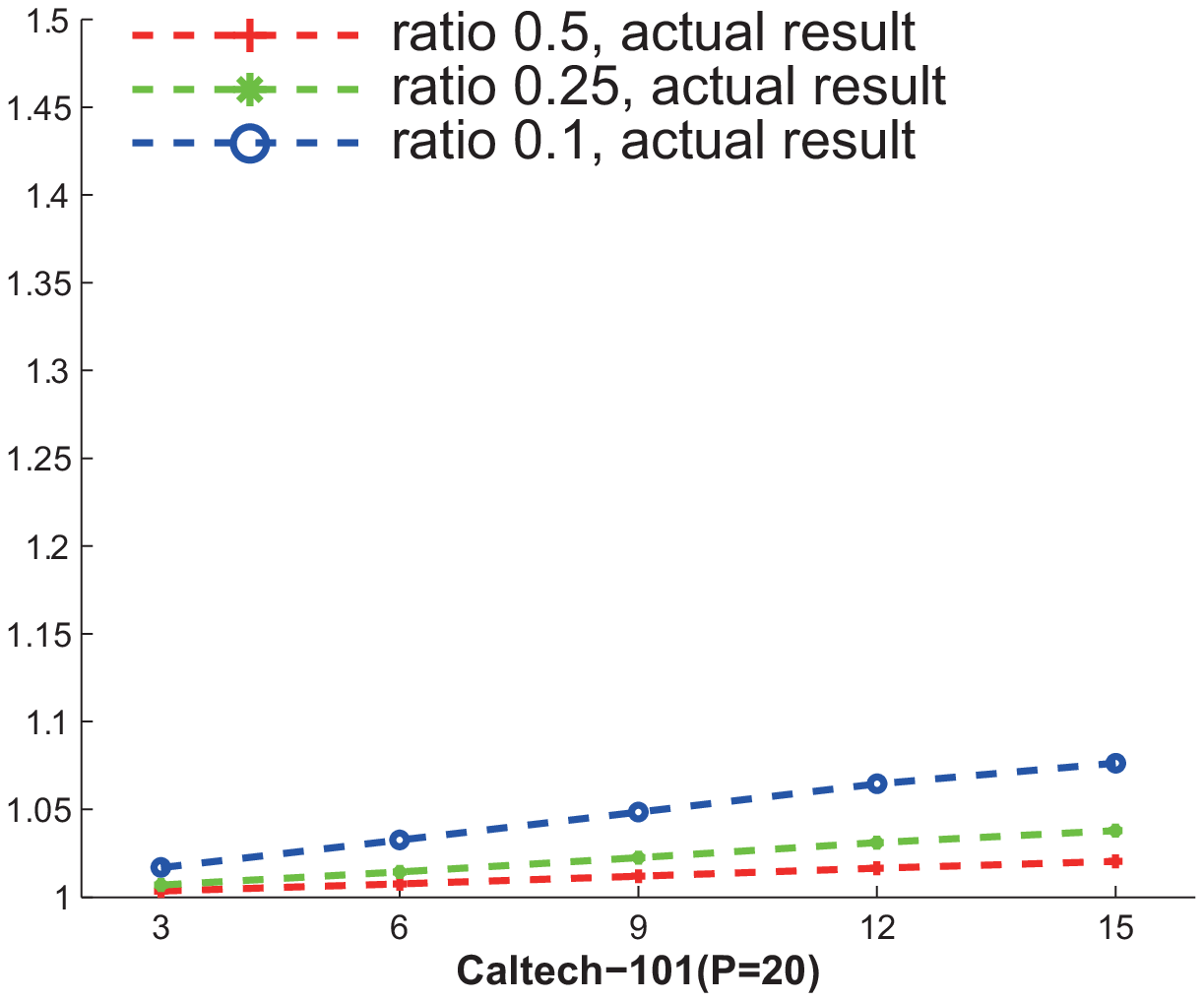}\\
\end{tabular}
\end{figure}
\begin{figure}
\noindent\begin{tabular}{cc}
  \includegraphics[width=0.245\textwidth]{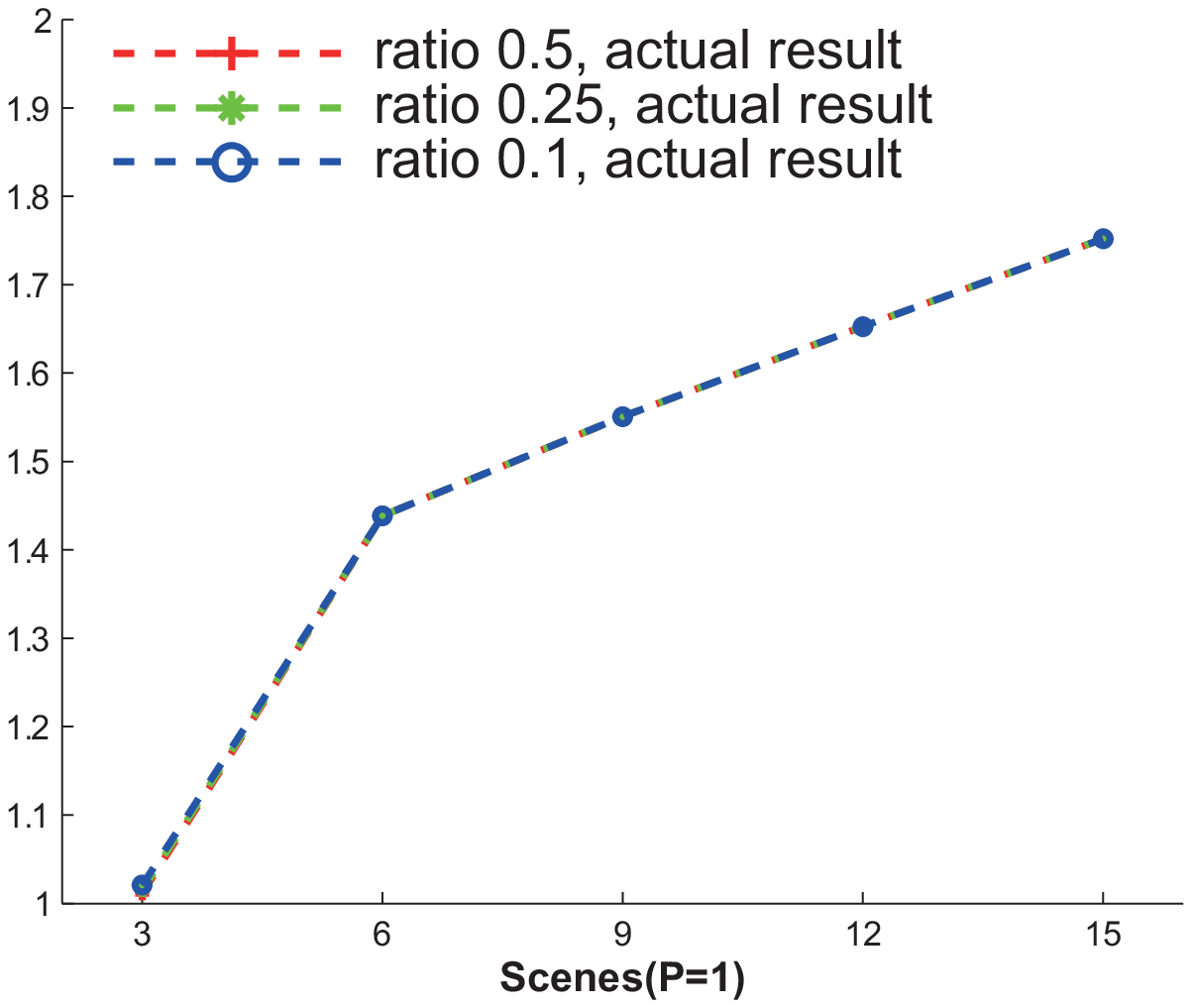}&
  \includegraphics[width=0.245\textwidth]{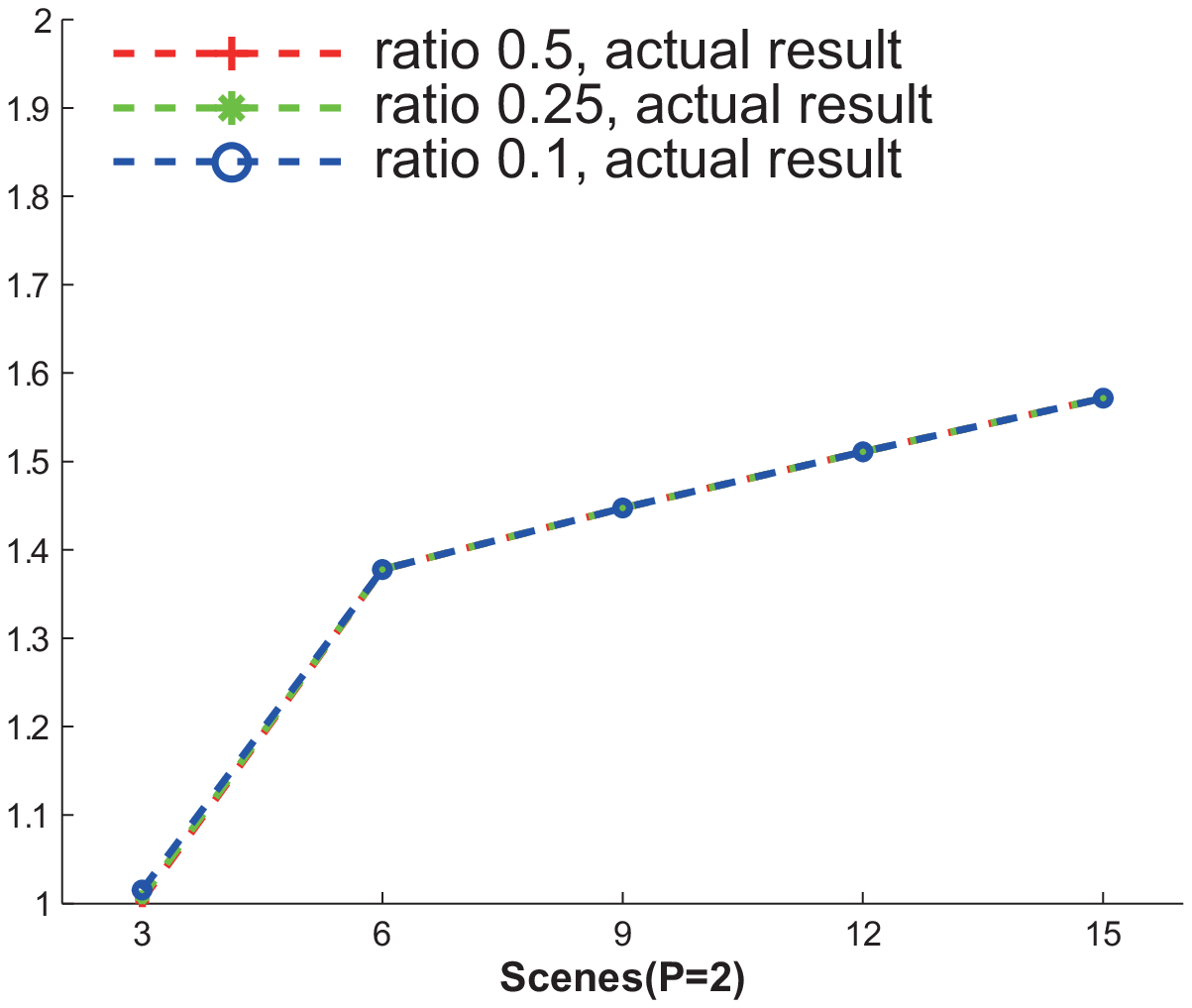}\\
  \includegraphics[width=0.245\textwidth]{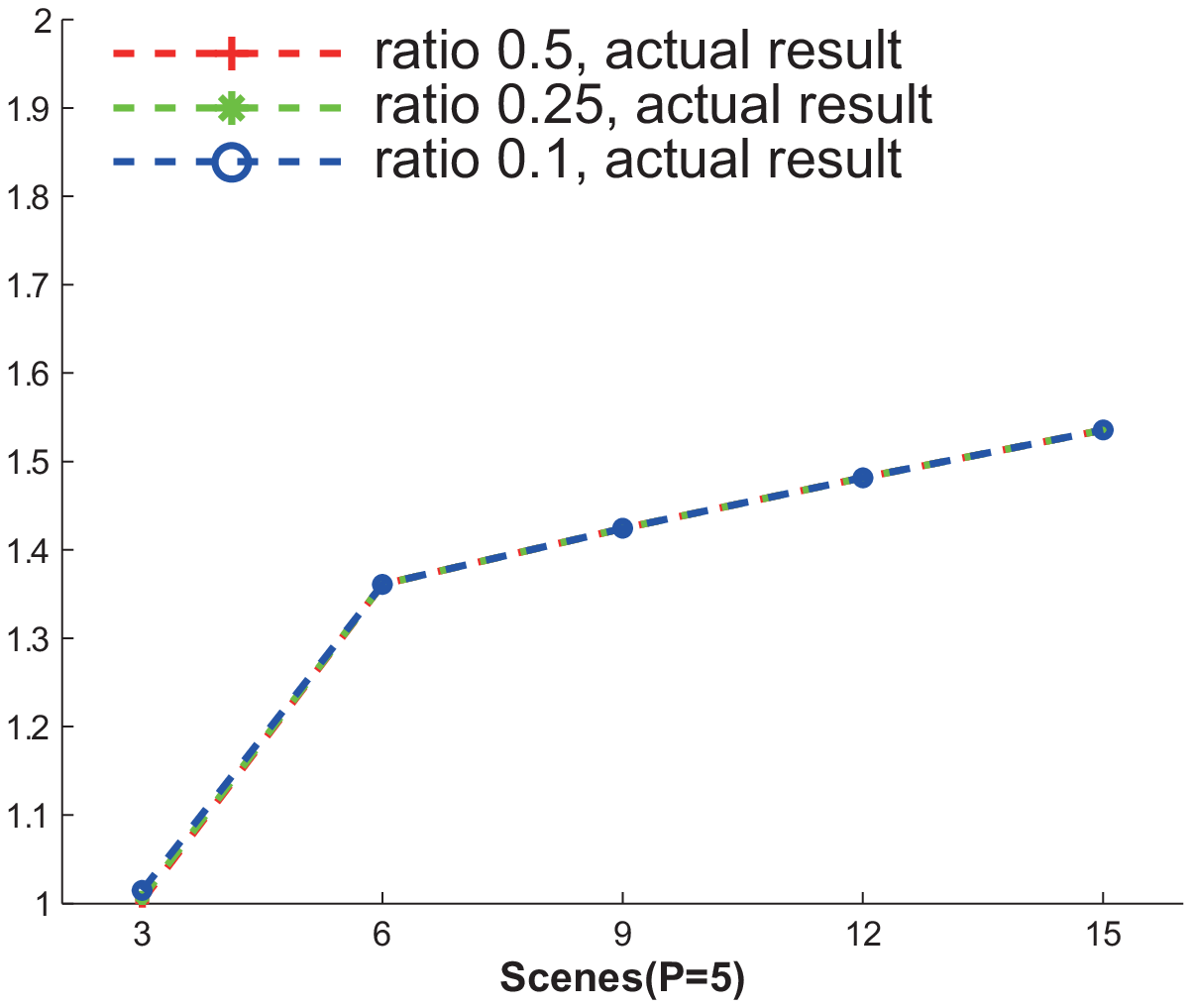}&
  \includegraphics[width=0.245\textwidth]{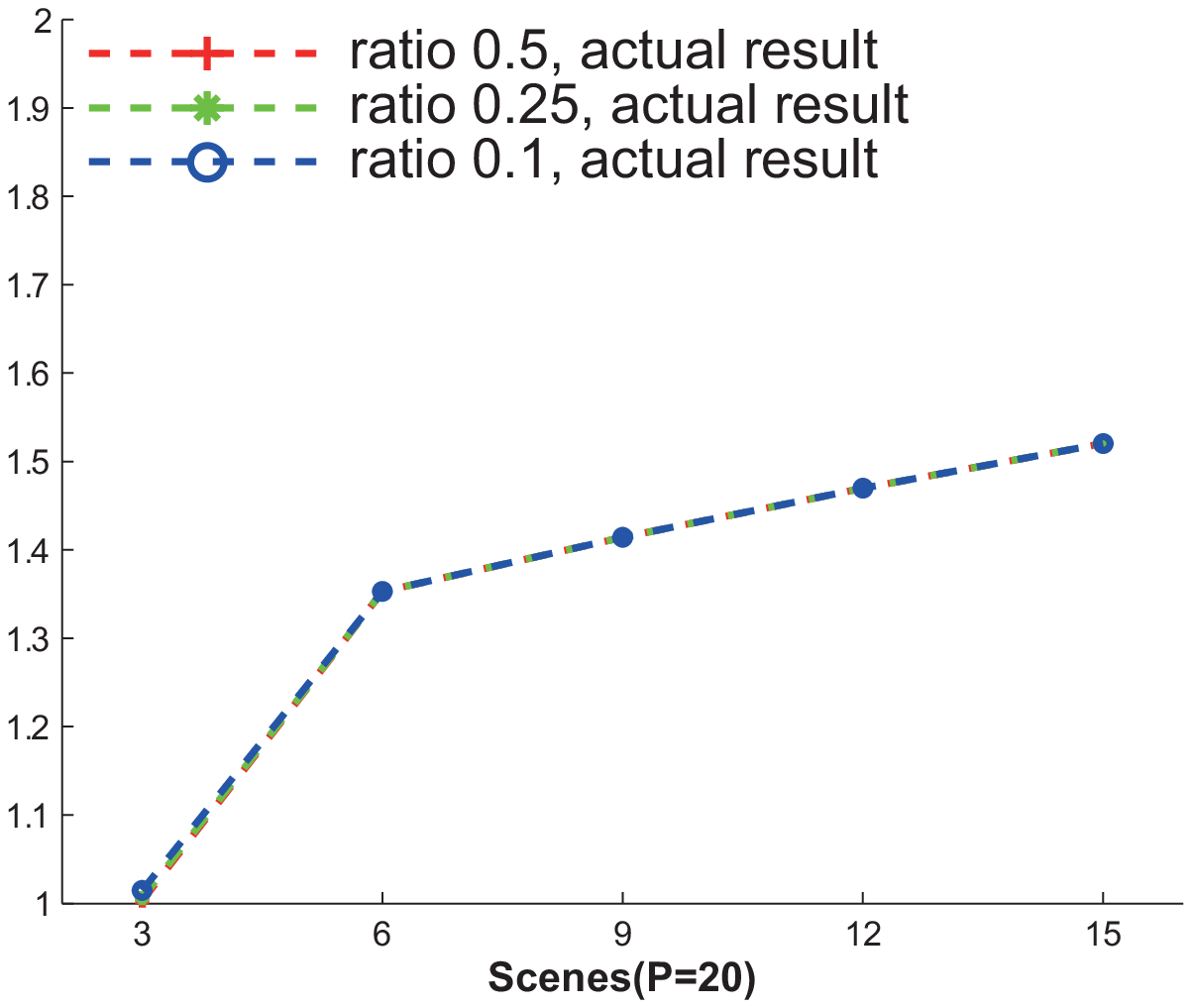}\\
  \includegraphics[width=0.245\textwidth]{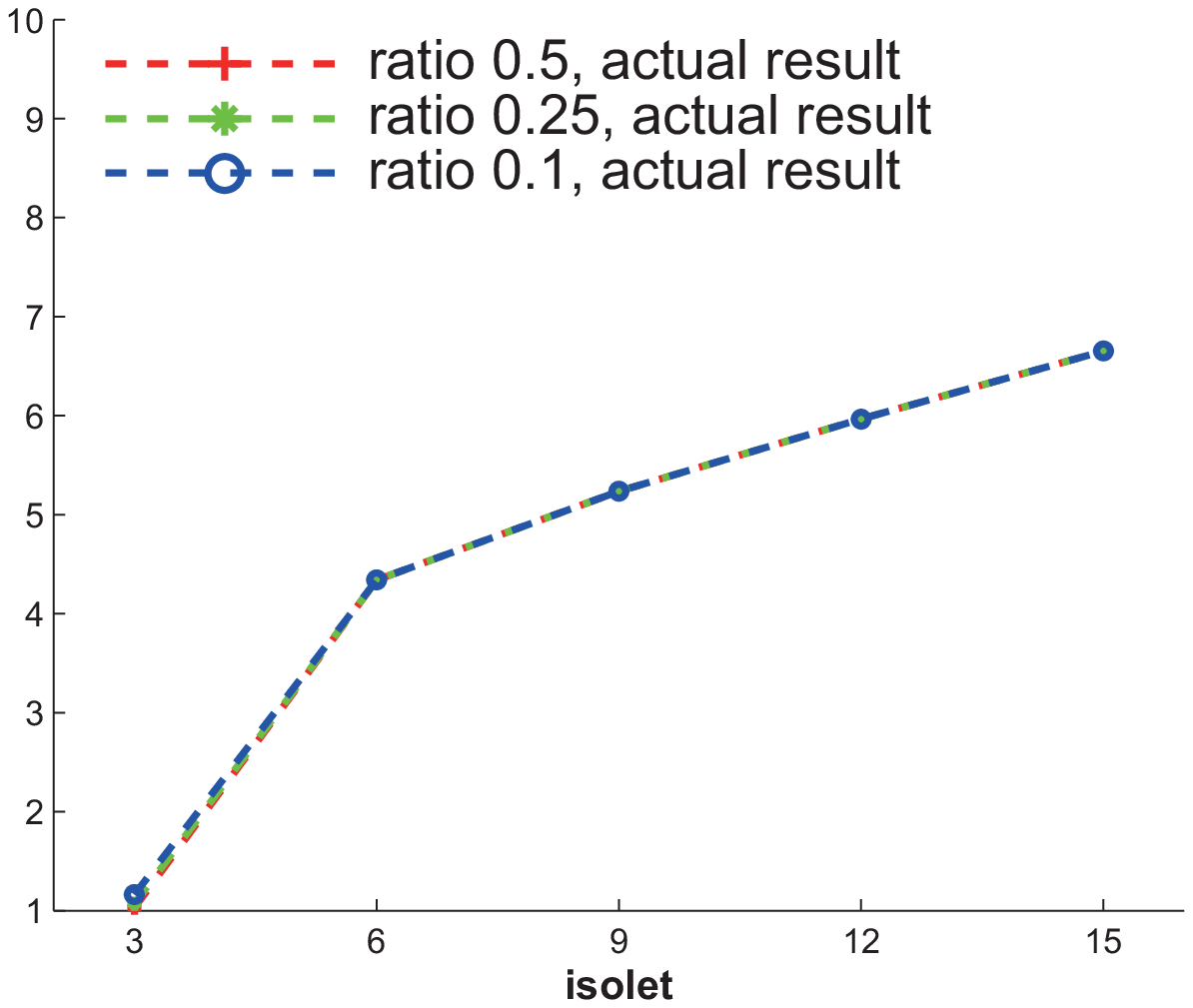}\\
\end{tabular}
\centering
\caption{Relative error v.s. projection dimension. The results of actual relative errors.} \label{experimentresult2}
\end{figure}

\section{Conclusions}
To the best of our knowledge, we proposed here the first non-trivial distributed protocol for the problem of computing a low rank approximation of a general function of a matrix. Our empirical results on real datasets imply that the algorithm can be used in real world applications. Although we only give additive error guarantees, we show the hardness of relative error guarantees in the distributed model we studied.

There are a number of interesting open questions raised by our work. Although our algorithm can work for a wide class of functions applied to a matrix, we still want to know whether there are efficient protocols for other functions which are not studied in this paper. Furthermore, this paper does not provide any lower bound for additive error protocols. It is an interesting open question whether there are more efficient protocols even with additive error.

\textbf{Acknowledgements} Peilin Zhong would like to thank Periklis A. Papakonstantinou for his very useful course \textit{Algorithms and Models for Big Data}. David Woodruff was supported in part by the XDATA program of the Defense Advanced Research Projects Agency (DARPA), administered through Air Force Research Laboratory contract FA8750-12-C-0323.



%



{
\bibliography{IEEEexample}
\bibliographystyle{unsrt}
}

\end{document}